\documentclass[12pt]{amsart}
\usepackage{amsmath,amscd,amsfonts,amssymb, color,bbm, bm,
braket, mathtools}
\usepackage{latexsym, graphicx, pstricks,rotating, enumerate}
\usepackage[margin=1in]{geometry}
\usepackage[colorlinks = true, linkcolor = blue, urlcolor  = blue, citecolor = red]{hyperref}
\usepackage[capitalise]{cleveref}
\definecolor{darkblue}{rgb}{0.0,0.0,0.3}
 \hypersetup{colorlinks=true,citecolor=darkblue,
 unicode=true,pdftitle=markovian-inf.pdf,pdfauthor=Ritabrata
 Sengupta, urlcolor=darkblue} 
\numberwithin{equation}{section}

\newtheorem{defin}{Definition}[section]
\newtheorem{rem}{Remark}[section]
\newtheorem{eg}{Example}[section]
\newtheorem{theorem}{Theorem}[section]
\newtheorem{lemma}{Lemma}[section]

\newcommand{\tr}{\mathrm{Tr\,}}
\newcommand{\dd}{\mathrm{d}}

\newcommand{\ketbra}[1]{\ensuremath{\left|#1\right\rangle
\hspace{-3pt}\left\langle #1\right|}}
\usepackage{mathrsfs}

\renewcommand{\epsilon}{\varepsilon}

\begin{document}
\title{Infinite dimensional dynamical maps}
\author{Bihalan Bhattacharya}
\address{Institute of Physics, Faculty of Physics, Astronomy and Informatics, Nicolaus Copernicus University, Grudziadzka 5/7, 87-100 Torun, Poland}
\email[Bihalan Bhattacharya]{\href{mailto:bihalan@gmail.com }{bihalan@gmail.com }}
\author{Uwe Franz}
\address{Laboratoire de math{\'e}matiques de Besan{\c c}on, Universit{\'e} de Franche-Comt{\'e}, 16, Route de Gray, F-25 030 Besan{\c c}on cedex, France}
\email[Uwe Franz]{\href{mailto:uwe.franz@univ-fcomte.fr }{uwe.franz@univ-fcomte.fr }}
\author{Saikat Patra}
\address{Department of Mathematical Sciences, Indian Institute of
Science Education \& Research (IISER) Berhampur, Transit campus,
Berhampur, 760 010, Ganjam, Odisha, India.}
\email[Saikat Patra]{\href{mailto:saikatp@iiserbpr.ac.in}{saikatp@iiserbpr.ac.in}}
\author{Ritabrata Sengupta}
\address{Department of Mathematical Sciences, Indian Institute of
Science Education \& Research (IISER) Berhampur, Transit campus,
Berhampur, 760 010, Ganjam, Odisha, India.}
\email[Ritabrata Sengupta]{\href{mailto:rb@iiserbpr.ac.in}{rb@iiserbpr.ac.in}}
\begin{abstract}
 
\par Completely positive trace preserving maps are widely used in quantum information theory. These are mostly studied using the master equation perspective. A central part in this theory is to study whether a given system of dynamical maps $\{\Lambda_t: t \ge 0\}$ is Markovian or non-Markovian. We study the problem when the underlying Hilbert space is of infinite dimensional. We construct a sufficient condition for checking P (resp. CP) divisibility of dynamical maps. We construct several examples where the underlying Hilbert space may not be of finite dimensional. We also give a special emphasis on Gaussian dynamical maps and get a version of our result in it.    

\smallskip

\end{abstract}
\maketitle


\section{Introduction}
 \par  The study of dynamical stochastic maps is one of the central concepts in probability theory. The non-commutative analogue, written in the language of quantum probability is also important with the increasing advancement of quantum information theory. Quantum information theory is the study of the achievable limits of information processing within quantum mechanics. The theory of open quantum systems is the backbone of nearly all modern research in quantum statistical mechanics and its applications.  The description of quantum systems interacting with their environment is the central objective of the theory of open quantum systems. 
 
\par  The theory of the open quantum system plays an important role in many applications of quantum physics. Quantum theory introduces a deterministic law, the Schr\"odinger equation, which governs the evolution of the quantum states which are non-commutative analogue of probability distribution. The actual dynamics of any open quantum systems are expected to deviate to some extent from the idealised Markovian evolution, which appears from the condition of weak coupling to a memoryless reservoir or an isolated system. The reason for this deviation has been physically explained as the back-flow of information from the environment to the system. This point of view was studied in Verstraete, et al. \cite{FC-1}, which further suggests that a properly engineered non-Markovian evolution can improve the efficiency of quantum technology protocols. As a result such non-Markovian evolutions has been used in quantum technologies, in particular as a quantum memory as well as quantum cryptography \cite{Vasile}. Such non-Markovian channels are also being  used as resources in quantum information processing \cite{Resource_1, Resource_2, MR4209138, berk}.   Hence it has become an important question to identify whether a given one parameter family of quantum evolution is Markovian or not. We call such family as \emph{dynamical maps} and give a more formal definition later (\cref{def1}).

\par In classical probability theory, Markov process is well studied.   In quantum theory, dynamics of a system is described by dynamical maps which is a family of completely positive trace preserving linear maps acting on complex separable Hilbert space.  There are multiple approaches to address quantum non-Markovianity.  It precisely depends on the context under study. The hierarchy between various definitions of non-markovianity has been studied in literature \cite{Hall2018}.  Among them there is an approach by Breuer-Laine-Piilo (BLP)\cite{BLP2009} which considers the change in distinguishability of quantum states during the dynamics. If there is no backflow of information, the distinguishabilty of states decreases during the dynamics. This is a signature of Markovianity. Another approach is by Rivas-Huelga-Plenio (RHP) \cite{RHP} in which the dynamics is expressed in terms of the concatenation of positive maps. This gives rise to the idea of propagator of the dynamics and the notion of Complete positive (CP) divisibility. CP divisibility of a dynamical map is often considered as quantum Markovianity. This approach is known as divisibility approach .  If the environment is large compared to the system and assumed to be stationary i.e. there is no change in the environment, and if the coupling between system and environment is week, it can be shown that the Born-Markov approximation \cite{bru3} leads
to the complete positive (CP) divisibility of the dynamics\cite{Breuer_review,Alonso_review,Rivas2014,Darek2022}. On the contrary, beyond the limit of weak coupling and
large stationary environment, there is no validation of this approximation and CP-divisibility may break down. This causes non-Markovian information backflow from the environment to the system. CP-divisibility breaking is  necessary to have information backflow. This implies whenever distinguishability of system state increases, it implies the information is coming back from the environment and the divisibility of the dynamics has been broken.

 \par Both of these approaches provide a quantification or measure of non-Markovianity. In general they do not agree. But it is important to mention that they are connected. The connection between these two approaches has been a line of study \cite{Rivas2011, Rivas2014,CRS-prl-18,Sagnik2019,Darek2022}.  It has been shown that these two measures can be reconciled. If we restrict ourselves in invertible dynamical maps, these two notions are equivalent. Recently Chru{\'s}ci{\'n}ski et al. have established the connection between these two notions for general non-invertible dynamical maps \cite{CRS-prl-18}.

\par Most of the articles mentioned above deals with evolutions for finite dimensional quantum systems. However a good portion of the theory and experimental works in quantum optics and communications deals with  infinite dimensional systems. In this case not much work, even theoretical, has been done to identify non-Markovian quantum system. In this paper we study this problem from a dimension independent point of view. In this way we generalise the notions given in the reference \cite{CRS-prl-18}. We start with basic definitions and a short introduction of the notations, viz. P-divisibility and CP-divisibility of dynamical maps  in \cref{s1}. Further we prove the following results : 
\begin{itemize}
\item A relation between contractivity and positivity of a map from trace class operators i.e. $\mathcal{B}_1(\mathcal{H})$ on the Hilbert space $\mathcal{H}$ to itself. (\cref{thm-pos=contr}).
\item A sufficiency condition for checking P-divisibility and CP-divisibility of a given dynamical map (\cref{th:diff}).
\end{itemize}

\par By using the result stated above we  give explicit examples of divisible/ indivisible family of maps in \cref{s2}. 
\begin{itemize}
\item In particular we give explicit examples of generic non-CP divisible maps in \S\ref{ss3.2}. 
\item This is followed by multi-parameter family of maps given as a combination of various previously known CP-maps \cref{ex-idemp}. 
\item We considered situations when the underlying Hilbert spaces of such families of maps need not be of finite dimensions. Based on the values of the parameters suitably chosen, the family can be P-divisible, CP-divisible, or may not be divisible at all (\cref{Theorem3.2} -- \ref{Theorem3.4}).    This culminates to further examples in \S\ref{ss3.4},  and \S\ref{ss3.6}. 

\item In \S\ref{ss3.7} we give a dynamical map given by a one-parameter Schur product. We show that this gives rise to a non-Markovian family of maps. Even though our map is defined in infinite dimensional systems, the proof technique can be applied for finite dimensional systems as well. 
\end{itemize}
\par A very special class of infinite dimensional channel is Gaussian channels. These channels map  Gaussian states to Gaussian states. In the \cref{gau} we deal with the P-divisibility property of such channels. For completeness we also give a short introduction of Gaussian channels based on Weyl operators. 
\begin{itemize}
\item We consider one parameter family of $n$-mode Gaussian channel and give explicit condition for its P-divisibility (\cref{th4.1}).  
\item We further construct explicit examples of not P-divisible Gaussian channels in $1$-mode (\cref{eg4.1}) and in $2$-modes (\cref{eg4.2}).
\end{itemize}
We also note that our method of detecting P-divisibility for  $n$-mode Gaussian channels is easier and more general than the method given in \cite{PhysRevLett.118.050401} which only deals with $1$-mode Gaussian channels with the same  assumptions. 
\par The paper is organised as follows. \cref{s1} gives basic definitions and results including the sufficiency theorem. In \cref{s2} we give various examples of divisible/indivisible dynamical maps. \cref{gau} deals with the case of Gaussian dynamical maps. For readability, we gave a self contained introduction of Gaussian channels. This section ends with examples of indivisible Gaussian dynamical maps. The paper ends with conclusion in \cref{s5} along with possible future directions. 


\section{Divisibility in infinite dimensional space}\label{s1}
  \par We begin with connecting quantum theory as a non-commutative version of classical probability theory. A (classical) probability space can be expressed in the Kolmogorov triplet format $(\Omega, \mathcal{E}, P)$, where $\Omega$ is the sample space, $\mathcal{E}$ is a $\sigma$-algebra on $\Omega$ called the set of events, and $P$ is the probability measure on it. In the same spirit, quantum systems can be expressed as a triplet $(\mathcal{H}, \mathcal{P(H)},\rho)$; where $\mathcal{H}$ denotes a separable complex Hilbert space of finite and infinite dimensions, $\mathcal{P(H)}$ is the set of all projections, and $\rho$ is a positive semidefinite trace class Hermitian operator with $\tr \rho=1$. $\rho$ is said to be the state of a particle or simply as a state. 
 \par Transformation of quantum systems is represented by a completely positive trace preserving map $\Lambda: \mathcal{B(H)} \to \mathcal{B(K)}$, which is also called as quantum channel. Representation of these maps was independently discovered by Sudarshan, Mathews, and Rau \cite{sud1}, Kraus \cite{kraus}, and Choi \cite{choi1}. This has several applications in quantum information theory -  for more details, we refer to the book of Nielsen and Chuang \cite{NC} or more recent one by Wilde \cite{wilde}. 

\par  To explain quantum evolution physically, we may use the following structure. Let us consider the Hilbert space of the system as $\mathcal{H}_S$ and the auxiliary space as $\mathcal{H}_E$. The state $\rho_S(0)$ is in  $\mathcal{H}_S$ at time $t=0$. A fixed state $\rho_E$ acting on $\mathcal{H}_E$ is used for dilation purpose. Let $H_S$ be the system Hamiltonian, $H_E$ be the Hamiltonian from the environment, and $H_{int}$ be the interaction Hamiltonian between system and environment. The effect of this on the whole system can be written as a global unitary 
 \[U(t) = \exp\left[-\frac{\imath}{\hbar} (H_S + H_E + H_{int})t\right].\]
 
The reduced dynamics can be written as 
\[\rho(t) = \tr_E [U(t) (\rho(0) \otimes \rho_E) U(t)^\dag].\]
The formulation is similar to the one mentioned earlier. In fact one can check that the reduced dynamics can be written in the form 
\[\rho(t) = \Lambda_t(\rho(0)) = \sum_j K_j \rho(0) K_j^\dag,\]
where $\sum_j K_j^\dag K_j = I$, the identity operator on $\mathcal{H}_S$. $K_j$'s are called the Kraus operators of the system. Evolution of the unitary operator in terms of \emph{master equation}, was given by Gorini, Kossakowski, and Sudarshan \cite{MR0406206}, and Lindblad \cite{lindblad}. A combined study of this from a quantum probability point of view can be found in \cite{krp4} and from thermodynamics point of view in \cite{bru3}.

\par Let $\mathcal{H}$ stands for a complex separable Hilbert space and let $\mathcal{B}(\mathcal{H})$ and $\mathcal{B}_1(\mathcal{H})$ be the spaces of bounded linear operators and trace class operators on $\mathcal{H}$ respectively.
\begin{defin} \label{def1}
A dynamical map $\{\Lambda_t\}_{t \ge 0}$ is a family of completely
positive (CP) and trace-preserving (TP) maps acting on the space 
$\mathcal{B}_1(\mathcal{H})$ of trace class operators on the Hilbert space
$\mathcal{H}$.  This family  $\{\Lambda_t\}_{t \ge 0}$ is divisible if 
\[\Lambda_t = \Phi_{t, s} \Lambda_s,\]
where $\Phi_{t, s}: \mathcal{B}_1(\mathcal{H}) \to \mathcal{B}_1(\mathcal{H})$ is a linear map for
every $t \ge s$. 
\end{defin}
\begin{defin}
    A dynamical map $\{\Lambda_t\}_{t \ge 0}$ is P-divisible if that map $\{\Lambda_t\}_{t \ge 0}$ is divisible and the linear map $\{\Phi_{t, s}: \mathcal{B}_1(\mathcal{H}) \to \mathcal{B}_1(\mathcal{H}):0\leq s \leq t\}$ can be chosen positive and trace preserving.
\end{defin}
\begin{defin}
    A dynamical map $\{\Lambda_t\}_{t \ge 0}$ is CP-divisible if that map $\{\Lambda_t\}_{t \ge 0}$ is divisible and the linear map $\{\Phi_{t, s}: \mathcal{B}_1(\mathcal{H}) \to \mathcal{B}_1(\mathcal{H}):0\leq s \leq t\}$ can be chosen  completely positive and trace preserving.
\end{defin}

\begin{theorem}\label{thm-pos=contr}
Let $\Phi$ be a Hermiticity and trace preserving linear map on
$\mathcal{B}_1(\mathcal{H})$. Then $\Phi$ is positive if and only if
$\Phi$ is contractive for all Hermitian operators $X \in
\mathcal{B}_1(\mathcal{H})$.
\end{theorem}

\begin{proof}
 Let us assume that $\Phi$ is positive and $X \in \mathcal{B}_1(\mathcal{H})$  be an arbitrary Hermitian operator. By spectral theorem we can write 
 \[X = X^{+} - X^{-}\]
where $X^{+}$ and $X^{-}$ are positive operators such that $X^{+}X^{-}=O$. Here $O$ stands for zero operator. So, $\vert X \vert = X^{+} + X^{-}$. Now,
\[ \| \Phi (X) \| _1 = \| \Phi (X^{+} - X^{-}) \|_1 = \| \Phi (X^{+}) -
\Phi(X^{-}) \|_1 = \tr | \Phi (X^{+}) - \Phi(X^{-}) | \]
Here $\| . \|_1$ stands for trace norm and ``$\tr(.)$" stands for
trace of an operator. We know that from triangle inequality,
\[ \| \Phi (X^{+}) - \Phi(X^{-}) \|_1 \leq \| \Phi (X^{+}) \|_1 + \|
\Phi (X^{-}) \| _1 \] 
Since $X^{+}$ and $X^{-}$ are positive operators and $\Phi$ is a trace preserving positive linear map, we have 
\begin{eqnarray*}
\| \Phi (X^{+}) \|_1 + \| \Phi (X^{-}) \|_1 &=& \tr [\Phi (X^{+})] +
\tr [\Phi (X^{-})] \\
&=& \tr[X^{+}] + \tr[X^{-}]\\
&=& \tr[(X^{+}+X^{-})] \\
&=& \tr | X | \\
&=& \| X \|_1.
\end{eqnarray*}
Hence we have shown that $\| \Phi (X) \| _1 \leq \| X \|_1$ i.e $\Phi$
is contractive.

\par Conversely, let $\Phi$ be contractive on
$\mathcal{B}_1(\mathcal{H})$.  Let $X \in \mathcal{B}_1(\mathcal{H})$
be a positive operator  and $\Phi$ be trace preserving map  i.e. $\tr
[\Phi (X)]= \tr [X]$. Now
\[
\| X \|_1 = \tr |X| = \tr X = \tr [\Phi (X)] \leq \tr |\Phi (X)| = \|
\Phi (X)\|_1,
\]
as $X \in \mathcal{B}_1(\mathcal{H}) $ and   $ X \geq 0$ . 

\par Since  $\Phi$ is contractive for all Hermitian operators $ X \in \mathcal{B}_1(\mathcal{H})$,
\begin{eqnarray*}
\| \Phi (X) \|_1 \leq  \|  X \|_1
\end{eqnarray*}
Hence for any positive operator $X$, $\tr|\Phi(X)|=\| \Phi (X) \|_1 = \|  X \|_1 = \tr[X]  = \tr[\Phi (X)]$. By the Hermiticity preservation of $\Phi$, $Y=\Phi(X)$ is Hermitian and can be written as the difference of its positive and negative parts $Y_+$, $Y_-$.
Then
\[
\tr Y_+ + \tr Y_-=\tr|\Phi(X)|  = \tr[\Phi (X)] = \tr Y_+ - \tr Y_-,
\]
which implies $Y_-=0$ and therefore $\Phi(X)= Y = Y_+ \ge 0$.
\end{proof}

\begin{theorem} \label{th:diff}
Let $\{\Lambda_t\}_{t \geq 0}$ be a family of  dynamical maps.
\begin{enumerate}
\item 
If $\{\Lambda_t\}_{t \geq 0}$ is P-divisible, then the function $t\mapsto \|\Lambda_t (X)\|_1$ is non-increasing for any trace class Hermitian operator $X \in \mathcal{B}_1(\mathcal{H})$. If this function is differentiable, then
\begin{equation}\label{eq1}
\frac{\dd}{\dd t} \|  \Lambda_t (X) \|_1 \leq 0.
\end{equation} 
\item
If $\{\Lambda_t\}_{t \geq 0}$ is CP-divisible, then the function $t\mapsto \| \mathcal{I} \otimes \Lambda_t (Y)\|_1$ is non-increasing for any trace class Hermitian operator $Y \in \mathcal{B}_1(\mathcal{H}\otimes\mathcal{H})$. If this function is differentiable, then
\begin{equation}\label{eq2}
\frac{\dd}{\dd t} \| \mathcal{I} \otimes \Lambda_t (Y) \|_1 \leq 0.
\end{equation}
\end{enumerate}
If $\{\Lambda_t\}_{t \geq 0}$ is furthermore surjective, then the converses also hold.
\end{theorem}

Let us recall the following fact which follows from basic linear algebra.

\begin{lemma}
    A family of dynamical maps $ \{\Lambda_t \}_{t\geq 0}$ on $\mathcal{B}_1(\mathcal{H})$ is divisible if and only if $ \mathrm{Ker}(\Lambda_s)\subseteq
\mathrm{Ker}(\Lambda_{t})$ for $t>s$.
\end{lemma}
For a proof see, e.g, \cite[Chapter II, Section 2.4]{greub}. Note that in general $V_{s,t}$ in not uniquely determined and need not be bounded.

\begin{proof}{ of \cref{th:diff}}
(1)
Let $ \{ \Lambda_t \}_{ t \geq 0 }$  be P-divisible. This implies,
\[ \Lambda_{t} = \Phi_{t, s} \Lambda_s, \qquad t > s,\]
where $\Phi_{t, s}$ is a positive and trace preserving map on $\mathcal{B}_1(\mathcal{H})$. Now,
\[
\| \Lambda_{t}(X) \|_1 = \| \Phi_{t, s} \Lambda_s (X) \|_1 \leq \| \Lambda_s (X) \|_1
\]
for all Hermitian $X$ as $\Phi_{t, s} $ is positive and trace
preserving. So, $\{ \| \Lambda_t (X) \|_1 \}_{t \geq 0}$ is a
decreasing function and, if it is differentiable,  
\[\frac{\dd}{\dd t} \| \Lambda_t (X) \|_1 \leq 0
\]
holds.

\par Conversely, let us assume that  $t\mapsto \|\Lambda_t (X)\|_1$ is non-increasing for any Hermitian operator $X\in \mathcal{B}_1(\mathcal{H})$.  We know that a dynamical
map is divisible if and only if $\mathrm{Ker}(\Lambda_s)\subseteq
\mathrm{Ker}(\Lambda_{t})$. If
$\{\Lambda_t\}_{t \geq 0}$ is not divisible then there exist a $X\in
\mathcal{B}_1(\mathcal{H})$ such that  $\Lambda_s(X) = 0$, whereas
$\Lambda_{t}(X) \neq 0$ for some $ t>s $. Thus
$\|\Lambda_s(X)\|_1 < \|\Lambda_{t}(X)\|_1$ which contradicts the non-increasing nature of $\| \Lambda_t (X) \|_1 \leq 0$. So $\{\Lambda_t\}_{t \geq 0}$ is divisible.

\par Assume now that $\{\Lambda_t\}_{t \geq 0}$ is surjective and not P-divisible. Then there exists some $t > s $ such that 
\[\Lambda_{t} = \Phi_{t, s} \Lambda_s,  \]
where $\Phi_{t, s} $ is not positive. By Theorem \ref{thm-pos=contr} and the surjectivity of $\Lambda_s$, there exists $X\in\mathcal{B}_1(\mathcal{H})$ such that 
\begin{equation}
\| \Lambda_{t}(X) \|_1 = \| \Phi_{t, s}\Lambda_{s}(X) \|_1 > \|\Lambda_{s}(X)\|_1,
\end{equation}
which contradicts the non-increasing nature of  $t\mapsto \|\Lambda_t (X)\|_1$. Therefore $\Phi_{t,s}$ is positive and hence $\{ \Lambda_t\}_{t \geq 0}$ is P-divisible.

\par (2) Let $ \{ \Lambda_t \}_{ t \geq 0 }$  be CP divisible. This implies,
\[ \Lambda_{t} = \Phi_{t, s} \Lambda_s \quad t > s.\]
where $\Phi_{t, s}$ is a completely positive and trace preserving map on $\mathcal{B}_1(\mathcal{H})$. Now,
\begin{eqnarray*}
\|(\mathcal{I} \otimes \Lambda_{t})(Y) \|_1 = \| (\mathcal{I}
\otimes\Phi_{t, s}\Lambda_{s})(Y) \|_1 &=& \|(\mathcal{I}
\otimes\Phi_{t, s})(\mathcal{I} \otimes\Lambda_{s})(Y) \|_1\\
& \leq & \|(\mathcal{I} \otimes\Lambda_{s})(Y)\|_1
\end{eqnarray*}
for all hermitian trace class operators $Y$ on product space. As
$\Phi_{t, s} $ is completely positive and trace preserving.
$\mathcal{I} \otimes \Lambda_t$   is positive. So $\{ \|(\mathcal{I}
\otimes\Lambda_t)(Y) \|_1 \}_{t \geq 0}$ is a decreasing function and
$\frac{d}{dt} \| (\mathcal{I} \otimes \Lambda_t )(Y) \|_1 \leq 0$, if this function is differentiable.

\par Conversely, assume that $t\mapsto \|(\mathcal{I} \otimes \Lambda_t) (Y)\|_1$ is non-increasing for all hermitian trace class operator $Y\in\mathcal{B}_1(\mathcal{H}\otimes\mathcal{H})$.  Then, similarly as in Part (1), we can show that $\{(\mathcal{I} \otimes\Lambda_t)\}_{t\geq 0} $ is divisible, since 
\[(\mathcal{I} \otimes\Lambda_{t})=(\mathcal{I} \otimes\Phi_{t, s})(\mathcal{I} \otimes\Lambda_s)=(\mathcal{I} \otimes \Phi_{t,s} \Lambda_s)
\qquad \Leftrightarrow \qquad \Lambda_{t} = \Phi_{t,s} \Lambda_s \]
where $t >s$ and $\Phi_{t,s}:\mathcal{B}_1(\mathcal{H})\rightarrow
\mathcal{B}_1(\mathcal{H})$ is a linear map. So $\{\Lambda_t\}_{t\geq
0}$ is divisible.
\par Let us assume that $\{\Lambda_t\}_{t \geq 0}$ is surjective and not CP-divisible. Then, there exists some $t > s $ such that 
\[\Lambda_{t} = \Phi_{t, s} \Lambda_s  \]
where $\Phi_{t, s} $ is not completely positive. By Theorem \ref{thm-pos=contr} and the surjectivity of $\Lambda_s$, which is equivalent to the surjectivity of $\mathcal{I} \otimes\Lambda_s$, there exists $Y\in\mathcal{B}_1(\mathcal{H}\otimes\mathcal{H})$ such that 
\begin{eqnarray*}
\|(\mathcal{I} \otimes \Lambda_{t})(Y) \|_1 = \| (\mathcal{I}
\otimes\Phi_{t, s}\Lambda_{s})(Y) \|_1 &=& \|(\mathcal{I}
\otimes\Phi_{t, s})(\mathcal{I} \otimes\Lambda_{s})(Y) \|_1\\
& >& \|(\mathcal{I} \otimes\Lambda_{s})(Y)\|_1
\end{eqnarray*}
which contradicts the initial assumption of not-increasing property of the map   $t\mapsto \|(\mathcal{I} \otimes \Lambda_t)Y\|_1$.
\end{proof}

\section{Examples of divisible and not-divisible channels} \label{s2}
\par We use the above methods to check divisibility/non-divisibility of some classes of dynamical maps. We start with a simple example of unitary evolution. 
\subsection{One parameter unitary evolution}
Let  $\mathcal{H}$ be a separable Hilbert space.  Consider a family of Dynamical maps $\{\Lambda_t :t\geq 0\}$ on $\mathcal{B}_1(\mathcal{H})$ defined by unitary adjoining 
\[\Lambda_t(X)=\text{Ad}_{U_t}(X)= U_tX{U_t}^*.\]
$\{U_t\}$ is a time $t\ge 0$ dependent semigroup of unitary operators on $\mathcal{H}$. Since $\|\cdot \|_1$ is unitary invariant,  $ \| \Lambda_t (X) \|_1 = \|X\|_1 $ which is  $t$ independent. Hence 
      \[
        \frac{\dd}{\dd t}  \| \Lambda_t (X) \|_1 = \frac{\dd}{\dd t}  \| X \|_1=0. 
     \]
As $\{\Lambda_t\}$ is surjective, hence by  \cref{th:diff} $\{\Lambda_t\}$ is P-divisible. 
    
   Choose $ \mathcal{B}_1(\mathcal{H}\otimes \mathcal{H}) \ni Y=\sum_{i=1}^\infty Y_i^1\otimes Y_i^2$, $Y_i^1$  and  $Y_i^2 \in \mathcal{B}_1(\mathcal{H})$ where convergence is under strong operator topology. We write
     \begin{eqnarray*}
      \| (\mathcal{I} \otimes \Lambda_t) (Y) \|_1 &=& \left\|(\mathcal{I} \otimes \Lambda_t) \left(\sum_{i=1}^{\infty} Y_i^1\otimes Y_i^2 \right)\right\|_1\\
      &=&\left\| \sum_{i=1}^{\infty} Y_i^1\otimes\Lambda_t (Y_i^2 )\right\|_1\\
     &= &\left\| \sum_{i=1}^{\infty} Y_i^1\otimes\ U_t Y_i^2U_t^* \right\|_1  \\
     &=&\left\| (I_{\mathcal{H}}\otimes U_t) \left(\sum_{i=1}^{\infty} Y_i^1\otimes Y_i^2 \right) (I_{\mathcal{H}}\otimes U_t^{*})\right\|_1 \\
     &=&\left\| \sum_{i=1}^{\infty} Y_i^1\otimes Y_i^2 \right\|_1 \\
    &=&\left\| Y \right\|_1,
         \end{eqnarray*}
where, $I_{\mathcal{H}}\in \mathcal{B}(\mathcal{H})$ is the identity operator on $H$. This is $t$ independent and hence 
 \[
         \frac{\dd}{\dd t} \left\| (\mathcal{I} \otimes \Lambda_t) (Y) \right\|_1=0.
 \]
 As  $\{\Lambda_t\}$ is surjective, it is CP-divisible by \cref{th:diff}.

\begin{rem}
    It is easy to see that here we have to choose the family of maps $\{\Phi_{t, s}: \mathcal{B}_1(\mathcal{H}) \to \mathcal{B}_1(\mathcal{H}):0\leq s \leq t\}$ as follows
    \[\Phi_{t, s}(X)= U_t{U_s}^*XU_s{U_t}^* ,\]
    for which the family of dynamical maps $\{ \Lambda_t: \mathcal{B}_1(\mathcal{H}) \to \mathcal{B}_1(\mathcal{H}): t\geq 0 \}$ satisfy the relation
    \[\Lambda_t = \Phi_{t, s} \Lambda_s, \qquad \text{where} \quad  0\leq s \leq t.\]
\end{rem}
\subsection{Example of a generic non CP-divisible dynamical map} \label{ss3.2}
\par Consider a Countable dense subspace $\mathcal{A}=\{\ket{i}\hspace{-3pt}\bra{j}:i,j=0,1,2,\ldots,n,\ldots \}$ of $\mathcal{H}$, where $\ket{i}$ is a column vector whose $(i+1)$th entry is $1$ and rest of the entries of the sequence are zeros. Define  Kraus operators,
\begin{eqnarray*}
 E_1 &=& \ket{0}\hspace{-3pt}\bra{1}+\ket{1}\hspace{-3pt}\bra{0}+ \sum_{i=2}^{\infty} \ketbra{i},   \\
 E_2 &=& \frac{1}{\sqrt{2}}\sum_{i,j=0}^1 \ket{i}\hspace{-3pt}\bra{j}+ \sum_{i=2}^{\infty} \ketbra{i},\\
 E_3 &=& \frac{1}{\sqrt{2}}\sum_{i,j=0}^1 (-1)^{ i+j\equiv 0\text{mod}{2}}\ket{i}\hspace{-3pt}\bra{j}+ \sum_{i=2}^{\infty} \ketbra{i}.
\end{eqnarray*}
Define the family of maps $\{\Lambda_t : \mathcal{B}_1(\mathcal{H}) \rightarrow \mathcal{B}_1(\mathcal{H}) : t\geq 0\} $ for some separable Hilbert Space $\mathcal{H}$ as
\begin{equation}
  \Lambda_t \left(X \right) = \eta_t E_1 X E_1^* + \frac{\kappa_t}{2} E_2 X E_2^* + \frac{\kappa_t}{2} E_3 X E_3^*,
\end{equation}
with two time-dependent functions $\eta_t$ and $\kappa_t$ such that $\eta_t , \kappa_t \geq 0$ and $\eta_t+\kappa_t = 1$.
\par To check the non-Markovianity of this system of dynamical maps, take the unnormalised states  $\ket \psi = \ket{00}+\ket{11}$ and $\ket \phi = \ket {01} + \ket {10}$. Define  $Y= \ketbra{\psi} - \ketbra{\phi}$. Note that $\left\| (\mathcal{I} \otimes \Lambda_{t} ) Y \right\|_1= 4 \eta_t$. Choose  $\eta_t = t$. Then, 
\[ \frac{d}{dt} \left\| (\mathcal{I} \otimes \Lambda_t) (Y) \right\|_1 = 4>0  \]
for some neighbourhood of $t>0$ . This shows that $\{\Lambda_t\}_{t\ge0}$ is not CP divisible by \cref{th:diff}.

\subsection{Dynamical maps defined from an idempotent map}\label{ex-idemp}
\par We start with finite-dimensional Hilbert spaces, where the (complete) positivity of our dynamical maps can be defined  by known criteria, see, e.g., \cite{skst12}. Letting the dimension go to infinity in the next step, we can dynamical maps whose positivity properties follow from the closedness of the cones we study, or from arguments and results in, e.g., the paper \cite{fried19}.

For $ k\ge 1$, define $\mathcal{H}_i\cong\mathbb{C}^k$, with an orthonormal basis $\{e_{i,j}, j=1,\ldots,k\}$. Define $\mathcal{H}=\bigoplus_{i-1}^n \mathcal{H}_i \cong \mathbb{C}^{nk}$. Denote the matrix units with respect to  the orthonormal basis  by $E_{(i,j),(r,s)}$, $i,\,r=1,\ldots,n$, $j,\,s=1,\ldots,k$.  We construct dynamical maps from the following idempotent channels on $\mathcal{H}$.
\begin{enumerate}
\item 
The \emph{dephasing channel} $D:\mathcal{B}(\mathcal{H})\to \mathcal{B}(\mathcal{H})$, $D(X)=\mathrm{Tr}(X) I_{\mathcal{H}}$, where $\mathrm{Tr}$ denotes the normalised trace, $\mathrm{Tr}(X)=\frac{1}{nk} \sum_{i=1}^n\sum_{j=1}^k \langle e_{i,j}, X e_{i,j}\rangle$. Its corresponding Choi matrix is given by, 
\[
C_D = \frac{1}{nk} I_\mathcal{H}\otimes I_\mathcal{H}.
\]
\item 
The \emph{block-wise dephasing channel} $B:\mathcal{B}(\mathcal{H})\to \mathcal{B}(\mathcal{H})$, 
\[D_B(X) = \frac{1}{n}\sum_{i=1}^k \mathrm{Tr}_{\mathcal{H}_i}(X) I_{\mathcal{H}_i},\]
 where $I_{\mathcal{H}_i}=\sum_{j=1}^k E_{(i,j),(i,j)}$ is the identity on the subspace $\mathcal{H}_i$  and 
 \[\mathrm{Tr}_{\mathcal{H}_i} (X) = \frac{1}{k} \sum_{j=1}^k \langle e_{i,j}, X e_{i,j}\rangle\]
is the trace there. It has Choi matrix of the form -                
\[
C_B = \frac{1}{nk} \sum_{i=1}^n I_{\mathcal{H}_i}\otimes I_{\mathcal{H}_i}.
\]
This has spectrum\footnote{Since $C_B^2=\frac{1}{nk}C_B$, $\mathrm{tr}(C_B)=k$.} $\sigma(C_B) = \left\{0,\dfrac{1}{nk}\right\}$, where $\dfrac{1}{nk}$ has multiplicity $nk^2$ and $0$ has multiplicity $n(n-1)k^2$.
\item
The \emph{conditional expectation} onto block-diagonal matrices: $E:\mathcal{B}(\mathcal{H})\to \mathcal{B}(\mathcal{H})$ defined as  
\[E(X) =\sum_{i=1}^n I_{\mathcal{H}_i} X I_{\mathcal{H}_i}.\]
It has Choi matrix
\[
C_E = \sum_{i=1}^n \sum_{j,s=1}^k E_{(i,j),(i,s)}\otimes E_{(i,j)(i,s)}.
\]
The spectrum\footnote{Since $C_E^2=kC_E$, $\mathrm{tr}(C_E)=nk$.} of this Choi matrix is $\sigma(C_E)=\{k,0\}$, with multiplicities of $k$ is $n$ and multiplicity of $0$ is $n(nk^2-1)$. 
\item 
The identity channel $I=\mathrm{id}_\mathcal{H}$, which has Choi matrix
\[
C_I = \sum_{i,r=1}^n \sum_{j,s=1}^k E_{(i,j),(r,s)}\otimes E_{(i,j),(r,s)},
\]
with spectrum $\sigma(C_I) = \{nk,0\}$ where multiplicity of $nk$ is $1$ and multiplicity of  $0$ is $n^2k^2-1$.
\end{enumerate}

%
%
%

\begin{theorem}\label{Theorem3.2}
 Let, $a,b,c,d\in\mathbb{R}$ and $\Phi(a,b,c,d) = aI + b E + cB + dD$ be defined as above. Then $\Phi$ is completely positive if and only if
\[
nka+kb+\frac{c+d}{nk} \geq 0 
\]
\[kb+\frac{c+d}{nk}\geq 0
\]
\[c+d\geq 0 \; \text{and} \; d\geq 0.
\]
\end{theorem}
\begin{proof}
 $\Phi$ is completely positive if and only if the Choi matrix  $ C_{\Phi}= (\mathcal{I}_{nk}\otimes \Phi)(\ket{\Psi}\bra{\Psi})$ is positive. Using linearity of the construction, we get the the Choi matrix of $\Phi$ is 
\[
    C_{\Phi}= a C_I + b C_E + c C_B + d C_D. 
\]
By suitable row-column transformation, the Choi matrix $ C_{\Phi}$ takes the form
\[\left[(\ketbra{\bm{1}_n})\otimes(a\ket{\bm{1}_n}+b\mathbb{I}_n)+\frac{c+d}{nk} \mathbb{I}_{nk}\right]\oplus \frac{c+d}{nk}\mathbb{I}_{nk(k-1)} \oplus \frac{d}{nk}\mathbb{I}_{nk^2(n-1)}
\]
Here $\ket{\bm{1}_k}=\underbrace{(1,1,\cdots,1)^T}_{k \text{ times}}$.
Eigenvalues of the first term, i.e.
\[ \bigg[\ketbra{\bm{1}_k}\otimes(a\ketbra{\bm{1}_n}+b\mathbb{I}_n)\bigg]+\frac{c+d}{nk} \mathbb{I}_{nk}\] 
are $\left(nka+kb+\dfrac{c+d}{nk}\right), \, \left(kb+\dfrac{c+d}{nk}\right), \, \left(\dfrac{c+d}{nk} \right) $ with multiplicity $1, \,(n-1), \, n(k-1)$. 
Hence eigenvalues of the Choi matrix $C_\Phi$  are $\left(nka+kb+\dfrac{c+d}{nk}\right), \left(kb+\dfrac{c+d}{nk}\right), \left(\dfrac{c+d}{nk}\right)$, and  $\dfrac{d}{nk} $ with multiplicity $1, \, (n-1), \, n(k^2-1), \, nk^2(n-1)$.

Therefore $\Phi$ is completely positive if and only if each of the eigenvalues is positive, from which, after simplification we get  the result. 
\end{proof}

\begin{theorem}\label{Theorem3.3}
Let, $a,b,c,d\in\mathbb{R}$ and $\Phi(a,b,c,d) = aI + b E + cB + dD$ be defined as above. If $\Phi$ 
is $2$-positive then
\[
2a+2b+\frac{c+d}{nk} \geq 0, \quad c+d\geq 0, \quad \text{and } \; d\geq 0.
\] 
\end{theorem}

\begin{proof}

Choose $\ket{\Psi} = \ket{00} +\ket{11}$.  After making suitable row and column transformation of $(\mathcal{I}_2\otimes \Phi)(\ket{\Psi}\bra{\Psi})$, we get,
   \[
        \left[(a+b)\ketbra{\bm{1}_2}+\frac{c+d}{nk} \mathbb{I}_{2} \right]\oplus \frac{c+d}{nk}\mathbb{I}_{2k-1} \oplus \frac{d}{nk}\mathbb{I}_{2k(n-1)}
   \]
Here $\ket{\bm{1}}_2=(1,1)^T$. Eigenvalues of the above operator are $2a+2b+\dfrac{c+d}{nk}, \,\dfrac{c+d}{nk}, \, \dfrac{d}{nk} $ with multiplicity $1, \, (2k-1), \, 2k(n-1)$. Therefore if $\Phi$ is  $2$-positive then the above conditions satisfy. 
\end{proof}

\begin{theorem}\label{Theorem3.4}
Let, $a,b,c,d\in\mathbb{R}$ with $a,b\leq 0$ and $\Phi(a,b,c,d) = aI + b E + cB + dD$ be defined as above. Then $\Phi$ is l-positive if and only if
\[b\|C_{E}\|_{S(l)}+ d+c+al\geq 0
\]
where $1\le \ell\le nk$.
\end{theorem}

To prove this theorem, we need to introduce concept of \emph{mapping cones}  as given in \cite{skst12, MR3012443}. Let $P_l$ be the set of all $l$-positive maps. A closed convex cone $\mathcal{C}$ in $P(\mathcal{H})$ is said to be a mapping cone  if $\Psi_1\circ \Phi \circ \Psi_2  \in \mathcal{C} $ whenever $\Phi \in \mathcal{C}$ and $\Psi_1,\Psi_2 $ are completely positive maps. If $\mathcal{C}$ is a mapping cone in $P(\mathcal{H})$ and $\mathcal{C}$ is symmetric then the cone of $\mathcal{C}$-positive map is a closed cone in $P(\mathcal{H})$ generated by maps of the form $\Phi \circ \Psi $ where $\Phi \in \mathcal{C}$ and $\Psi$ is a completely positive map. This cone is  denoted by $P_\mathcal{C}$. If $P_\mathcal{C}=\mathcal{C}$ then $P_{P_l}=P_l$, where $P_{P_l}$ is the cone of $l$-positive maps. The dual cone of $\mathcal{C}$ is defined as 
\[\mathcal{C}^{\circ}=\{\Phi \in P(\mathcal{H}):\tr (C_{\Phi}C_{\Psi})\geq 0, \forall \Psi \in \mathcal{C}\}\]
and double dual cone of $\mathcal{C}$ is defined as 
\[\mathcal{C}^{\circ \circ}=\{\Phi \in P(\mathcal{H}):\tr (C_{\Phi}C_{\Psi})\geq 0 \quad \forall \Psi \in \mathcal{C}^{\circ}\}.\]  
Using these notations it can be shown that  $P_{P_l}^{oo}=P_{P_l}$ and $P_{P_l}=P_l$ \cite{MR2545616}.
\begin{proof}
 Let, $\Phi$ be a $l$-positive map. Then by using above two conditions we get,
\begin{eqnarray}
\Phi \in P_l 
&\Leftrightarrow & \tr (C_{\Phi}C_{\Psi})\geq 0 \quad \forall \;\Psi \in P_{P_l}^{o} \nonumber  \\
 &\Leftrightarrow & \tr (C_{\Phi}C_{\Psi})\geq 0 \quad \forall \;C_{\Psi} \in J^{-1}(S_{P_l}) \nonumber \\
 &\Leftrightarrow & \inf_{C_{\Psi} \in J^{-1}(S_{P_l})} \tr (C_{\Phi}C_{\Psi}) \geq 0 \nonumber  \\
 &\Leftrightarrow & \inf_{C_{\Psi} \in J^{-1}(S_{P_l})} \tr ((a C_{I}+ b C_{E} + c C_{B}+ d C_{D})C_{\Psi}) \geq 0.  \label{eq3.4}
\end{eqnarray} 
The right hand side of the above expression can be simplified as follows - 
\begin{eqnarray*}
0 &\leq & \inf_{C_{\Psi} \in J^{-1}(S_{P_l})} \tr ((a C_{I}+ b C_{E} + c C_{B}+ d C_{D})C_{\Psi}) \\
& = & \inf_{C_{\Psi} \in J^{-1}(S_{P_l})} \tr ((a C_{I}C_{\Psi}+ b C_{E}C_{\Psi} + c C_{B}C_{\Psi}+ d C_{D}C_{\Psi}) \\
& = & \inf_{C_{\Psi} \in J^{-1}(S_{P_l})} a\tr (( C_{I}C_{\Psi})+ b\tr( C_{E}C_{\Psi}) + c \tr ( C_{B}C_{\Psi})+ d \tr ( C_{D}C_{\Psi}))  \\
& = & \inf_{C_{\Psi} \in J^{-1}(S_{P_l})} a\tr (C_IC_{\Psi})+\inf_{C_{\Psi} \in J^{-1}(S_{P_l})} b\tr( C_EC_{\Psi}) \\
&&+\inf_{C_{\Psi} \in J^{-1}(S_{P_l})}c\tr \left( \sum_{i=1}^n I_{\mathcal{H}_i}\otimes I_{\mathcal{H}_i}\right)C_{\Psi}+\inf_{C_{\Psi} \in J^{-1}(S_{P_l})} d\tr \left(I_{\mathcal{H}}\otimes I_{\mathcal{H}}\right)C_{\Psi}.
\end{eqnarray*} 
Since $\sum I_{\mathcal{H}_i}\otimes I_{\mathcal{H}_i}
 \leq I_{\mathcal{H}}\otimes I_{\mathcal{H}}
 $ and the number of $1$'s in diagonal entries  in $I_{\mathcal{H}}\otimes I_{\mathcal{H}}$  is more than that of $\sum I_{\mathcal{H}_i}\otimes I_{\mathcal{H}_i}$, we have  
 \[\inf_{A\geq 0} \tr A \left(\sum_{i=1}^n I_{\mathcal{H}_i}\otimes I_{\mathcal{H}_i}
\right)\leq \inf_{A\geq 0} \tr A (I_{\mathcal{H}}\otimes I_{\mathcal{H}}
).\] 
Using this in the above expression we get,
\begin{eqnarray*}
0&\leq & \inf_{C_{\Psi} \in J^{-1}(S_{P_l})} a\tr ( C_IC_{\Psi})+\inf_{C_{\Psi} \in J^{-1}(S_{P_l})} b\tr(C_EC_{\Psi}) \\
&&+\inf_{C_{\Psi} \in J^{-1}(S_{P_l})}c\tr ( I_{\mathcal{H}}\otimes I_{\mathcal{H}})C_{\Psi}+\inf_{C_{\Psi} \in J^{-1}(S_{P_l})} d\tr ( I_{\mathcal{H}}\otimes I_{\mathcal{H}})C_{\Psi}.
\end{eqnarray*}
As $a$ and $b$ are negative, we may assume $a=-\lambda $ and $b=-\mu$ for some $\lambda , \mu \geq 0$ which further further simplifies  the right hand side of the above expression as,
\[ \inf_{C_{\Psi} \in J^{-1}(S_{P_l})}-\lambda\tr ( C_{B}C_{\Psi})+\inf_{C_{\Psi} \in J^{-1}(S_{P_l})} -\mu \tr ( C_{D}C_{\Psi})+c+d. \]
Applying   $\inf (-A)= -\sup A$ for any abounded subset $A$ of $\mathbb{R}$ the above expression takes the form,
\begin{eqnarray*}
&=&-\sup_{C_{\Psi} \in J^{-1}(S_{P_l})}\lambda\tr ( C_{B}C_{\Psi})-\sup_{C_{\Psi} \in J^{-1}(S_{P_l})} \mu \tr ( C_{D}C_{\Psi})+c+d \\
&=& -\lambda\sup_{C_{\Psi} \in J^{-1}(S_{P_l})}\tr ( C_{B}C_{\Psi})-\mu \sup_{C_{\Psi} \in J^{-1}(S_{P_l})} \tr ( C_{D}C_{\Psi})+c+d\\
&=& a\|C_{I}\|_{S(l)}+b\|C_{E}\|_{S(l)}+ d+c. \quad  \text{by \cite{MR2683522}}. 
\end{eqnarray*}
Since $\|C_{I}\|_{S(l)}=\|I\|_{P_k}=\|Ad_{I_H}\|_{P_k}=\|I_H\|_{(l)}^2=l$, by  \cite{skst12}. Therefore
\[  \inf_{C_{\Psi} \in J^{-1}(S_{P_l})} \tr ((a C_{I}+ b C_{E} + c C_{B}+ d C_{D})C_{\Psi})=al+b\|C_{E}\|_{S(l)}+ d+c.\]
Therefore from \eqref{eq3.4},
\[ \Phi \in P_l 
\Leftrightarrow  b\|C_{E}\|_{S(l)}+ al+c+d\geq 0\]
\end{proof}

\begin{rem}
   For any $X\geq 0$ in $\mathcal{B}(\mathcal{H})$ where $\mathcal{H}$ is a finite dimensional Hilbert space. In \cite{MR2683522}  the $r$-th operator norm of $X$ is defined by,
    \[\|X\|_{S(r)}= \sup_{\ket{v}} \{ \bra{v}X\ket{v}:SR(\ket{v}) \leq r\}.\]
In particular, $\|X\|_{S(1)} = \sup_{\ket{v}} \{ \bra{v}X\ket{v}:SR(\ket{v}) \leq 1\}$ which is the local Spectral radius of X.
 \par Here the spectrum of the Choi matrix $C_E$ is $\sigma(C_E)=\{k,0\}$, with multiplicities $m(k)=n$, $m(0)=n(nk^2-1)$. This implies that $\|C_E\|_{S(1)}= k.$ Positivity of the map $\Phi$ defined in \cref{Theorem3.4} is equivalent to 
   \[bk+ a+c+d\geq 0.\]
\end{rem}

\par Using the above notations  we can define a dynamical map 
\[
\Lambda_t = a_tI + b_t E + c_t B + d_t D,
\]
where $a_t,\, b_t,\, c_t$, and $d_t$ are given positive functions.  The following theorem gives a sufficient condition for this dynamical map to be divisible and computes the map $\Phi_{s,t}$.\footnote{If $a_s,b_s,c_s,d_s$ are positive, than it is sufficient that $a_s$ is strictly positive.} 
\begin{theorem}\label{Theorem3.5}
Let $\Lambda_t$ be defined as above. If  $a_s\not=0$, $a_s+b_s\not=0$, $a_s+b_s+c_s\not=0$ and $a_s+b_s+c_s+d_s\not=0$, then we have
\[
\Lambda_t = \Phi(\alpha,\beta,\gamma,\delta) \Lambda_s
\]
with
\begin{eqnarray*}
    \alpha &=& \frac{a_t}{a_s}, \\
    \beta &=& \frac{a_s b_t - b_s a_t}{a_s(a_s+b_s)}, \\
    \gamma &=& \frac{(a_s+b_s)c_t - c_s (a_t+b_t)}{(a_s+b_s)(a_s+b_s+c_s)}, \\
    \delta &=& \frac{(a_s+b_s+c_s)d_t - d_s(a_t+b_t+c_t)}{(a_s+b_s+c_s)(a_s+b_s +c_s+d_s)}.
\end{eqnarray*}
\end{theorem}

The proof follows from the case $n=4$ of the following lemma.
\begin{lemma}
   For  $n\ge 1$,  let $p_1,\ldots,p_n$ be a decreasing family of idempotents, i.e., satisfying the relations
$p_i^2 = p_i$  and  $p_i p_j = p_jp_i = p_{\max(i,j)}$
    for $i,j=1,\cdots,n$.
    For $\bm{x}=(x_1,\cdots,x_n)\in\mathbb{R}^n$, we define 
        $p(\bm{x}) = \sum_{i=1}^n x_i p_i$.
    For  $\bm{x},\, \bm{y}\in\mathbb{R}^n, \, p(\bm{x}) p(\bm{y}) = p(\bm{z})$ where
    \begin{equation}\label{eq-nilp-prod}
    z_i = \sum_{j=1}^{i-1} x_j y_i + x_i y_i + \sum_{j=1}^{i-1} x_i y_j, \quad \text{for } i= 1,\cdots, n.
    \end{equation}
Furthermore, if $\bm{x},\, \bm{z}\in\mathbb{R}^n$ satisfy $\sum_{j=1}^i x_j \not=0$ for $i=1,\ldots,n$, then we have $p(\bm{z}) = p(\bm{y}) p(\bm{x})$  with
    \begin{equation}\label{eq-y-def}
    y_1 =\frac{z_i}{x_1} \qquad\text{ and }\qquad  y_i = \frac{z_i}{\sum_{j=1}^i x_j} - \frac{x_i\cdot \sum_{j=1}^{i-1}z_j}{\sum_{j=1}^{i-1} x_j \cdot \sum_{j=1}^i x_j} \qquad\text{ for }i=2,\cdots,n.
    \end{equation}
\end{lemma}
\begin{proof}
The proof of the Equation \eqref{eq-nilp-prod} follows directly from the previous calculations. 
\begin{eqnarray*}
    p(\bm{x}) p(\bm{y}) &=& \sum_{i=1}^n x_i p_i \sum_{j=1}^n  y_i p_j \\
    & =&  \sum_{i,j=1}^n x_i y_j p_{\max(i,j)} \\
    &=& \sum_{k=1}^n \left( \sum_{i=1}^{k-1} x_i y_k + x_k y_k +\sum_{j=1}^{k-1} x_k y_j \right) p_k.
\end{eqnarray*}

\par We will prove Equation \eqref{eq-y-def} in two steps. Let $\bm{x}, \, \bm{z}\in \mathbb{R}^n$ be such that  $\sum_{j=1}^i x_j \not=0$ for $i=1,\ldots,n$, and define $\bm{y}\in\mathbb{R}^n$ by Equation \eqref{eq-y-def}.

\begin{itemize}
    \item Step I:
We show by induction that
\begin{equation}\label{eq-step1}
y_1+\cdots + y_i = \frac{z_1+\cdots + z_i}{x_1+\cdots+x_i} \quad \text{for } i=1,\cdots,n.
\end{equation}
For $i=1$ this is the definition of the $\bm{y}$'s given in Equation \eqref{eq-y-def}. Assume Equation \eqref{eq-step1} is true for $i<n$. Then we have
\begin{eqnarray*}
    y_1 + \cdots +y_{i+1} &=& \frac{z_1+\cdots+z_i}{x_1+\cdots+x_i} +  \frac{z_{i+1}}{\sum_{j=1}^{i+1} x_j} - \frac{x_{i+1}\cdot \sum_{j=1}^{i}z_j}{\sum_{j=1}^{i} x_j \cdot \sum_{j=1}^{i+1} x_j} \\
    &=& \frac{\sum_{j=1}^i z_j \cdot \sum_{j=1}^{i+1} x_j + z_{i+1}\cdot \sum_{j=1}^i x_j - x_{i+1} \cdot \sum_{j=1}^i z_j}{\sum_{j=1}^{i} x_j \cdot \sum_{j=1}^{i+1} x_j} \\
    &=& \frac{\sum_{j=1}^{i+1} z_j}{\sum_{j=1}^{i+1} x_i}.
\end{eqnarray*}
    \item Step II: To conclude, we insert Equations \eqref{eq-y-def} and \eqref{eq-step1} into the RHS of Equation \eqref{eq-nilp-prod},
    \begin{eqnarray*}
        \sum_{j=1}^{i-1} x_j y_i + x_i y_i + \sum_{j=1}^{i-1} x_i y_j &=&
        \sum_{j=1}^i x_j \left( \frac{z_i}{\sum_{j=1}^i x_j} - \frac{x_i \cdot \sum_{j=1}^{i-1} z_j}{\sum_{i=1}^{i-1} x_j \cdot \sum_{j=1}^i x_j} \right) + x_i \frac{\sum_{j=1}^{i-1} z_i}{ \sum_{j=1}^{i-1}x_i} \\
        &=& \frac{1}{ \sum_{i=1}^{i-1} x_j} \left( z_i \sum_{j=1}^{i-1} x_j - x_i \sum_{j=1}^{i-1} z_j + x_i \sum_{j=1}^{i-1} z_j \right) \\
        &=& z_i.
    \end{eqnarray*}
    This  completes the proof.
\end{itemize}
    \end{proof}
Note that by letting $k$ and / or $n \to \infty$ we can also generate infinite-dimensional examples which we will show below.

 \subsection{A family of dynamical maps which is P-divisible but not CP-divisible.}\label{ss3.4}  Consider a family of map $\{\Lambda_t \}$ on $\mathcal{B}(\mathcal{H})$  where  $\mathcal{H}=\bigoplus_{i-1}^n \mathcal{H}_i \cong \mathbb{C}^{nk}$, defined as follows 
\[\Lambda_t = a_tI + b_t E + c_t B + d_t D.\]
If the parameter $a_t, b_t, c_t, d_t \in \mathbb{R}$ satisfies all the conditions from Theorem \ref{Theorem3.2} then $\{\Lambda_t \}$ is completely positive map and for trace preserving we need the condition $a_t + b_t  + c_t  + d_t =1 $. \\

For the divisibility of $\{\Lambda_t \}$ using  \cref{Theorem3.5}  we can derive $\alpha,\beta,\gamma,\delta$ such that 
\[\Lambda_t = \Phi(\alpha,\beta,\gamma,\delta) \Lambda_s.\]

Here we can choose $a_t,\, b_t, \, c_t$ and $d_t$ and $a_s, b_s, c_s, d_s$ in a way such that  $\alpha,\beta,\gamma,\delta$ satisfy all the conditions from  \cref{Theorem3.3} or \cref{Theorem3.4} but not \cref{Theorem3.2}. Then $\Phi(\alpha,\beta,\gamma,\delta)$ is positive but completely positive which implies family of dynamical maps $\{\Lambda_t \}$ is P-divisible but not CP-divisible.
\par In particular, letting $n$ goes to infinity and  for fixed $k$ we get an infinite-dimensional example which is P-divisible but not CP-divisible:

Choose $a_t, \, b_t, \, c_t,\,  d_t$ and $a_s,\, b_s,\, c_s, \,d_s$ in a way such that  
$ \alpha =0 $, $\beta < 0$, $ k\beta + \gamma + \delta \geq 0$. This implies $\alpha,\beta,\gamma,\delta$ satisfy the conditions from \cref{Theorem3.4} but not those from \cref{Theorem3.2}.

\begin{rem}
     If we choose  $a_t, \, b_t,\, c_t,\, d_t$ and $a_s, \,b_s,\, c_s,\, d_s$ in a way such that $ k\beta + \gamma + \delta +\alpha < 0 $, then by \cref{Theorem3.4} $\Phi(\alpha,\beta,\gamma,\delta)$ is not positive. Hence we get a dynamical map $\{\Lambda_t \}$ which is not P-divisible.

\end{rem}

\subsection{An example of  CP-divisible family.}\label{ss3.6}  Consider a family of map $\{\Lambda_t \}$ on $\mathcal{B}(\mathcal{H})$  where  $\mathcal{H}=\bigoplus_{i-1}^n \mathcal{H}_i \cong \mathbb{C}^{nk}$, define as follows 

\[\Lambda_t = a_tI + b_t E + c_t B + d_t D.\]

If the parameter $a_t, b_t, c_t, d_t \in \mathbb{R}$ satisfies all the conditions from  \cref{Theorem3.2} then $\{\Lambda_t \}$ is completely positive map and for trace preserving we need the condition $a_t + b_t  + c_t  + d_t =1 $. \\

Divisibility of $\{\Lambda_t \}$ follows from  \cref{Theorem3.5}  by derive $\alpha,\beta,\gamma,\delta$ such that 
\[\Lambda_t = \Phi(\alpha,\beta,\gamma,\delta) \Lambda_s.\]
Take $a_t, b_t, c_t, d_t$ and $a_s, b_s, c_s, d_s$ in a way such that  $\alpha,\beta,\gamma,\delta$ satisfy all the conditions from  \cref{Theorem3.2}. $\Phi(\alpha,\beta,\gamma,\delta)$ is completely positive. Hence, $\{\Lambda_t \}$ is  CP-divisible.

\subsection{Divisibility of family of maps defined by Schur product}\label{ss3.7}
Let  $\Lambda_t : \mathcal{B}(\ell_2(\mathbb{N}))\rightarrow \mathcal{B}(\ell_2(\mathbb{N}))$ be a family of maps defined by Schur product as 
\begin{equation}\label{eg-e1}
 \Lambda_t(X) = A_t\circ X,
 \end{equation}
where $\circ $ is given as the Schur product, i.e. if $A_t=[[a(t)_{ij}]]$ and $X=[[x_{ij}]]$ with respect to some basis, then $A_t \circ X = [[a_{ij} x_{ij}]] $. It is a well known fact that Schur product gives a positive map provided the matrix with which the product is taken is positive. In fact the map will then be completely positive. Let $\{\ket{n}: n \in \mathbb{N} \cup \{0\}\}$ be a set of orthogonal basis of $\ell_2(\mathbb{N})$.  We want to construct a time $t$ dependent quantum channel, which is a CPTP map. Let $A_t$ be defined as 
\[A_t = I_{\ell_2(\mathbb{N})} + t \sum_{j=0}^\infty \left(\ket{j}\hspace{-3pt}\bra{j+1} + \ket{j+1}\hspace{-3pt}\bra{j}\right),\]
where  $I_{\ell_2(\mathbb{N})}$ is the identity operator on $\ell_2(\mathbb{N})$. 
$A_t$ is a  self-adjoint operator. Define a sequence of $n\times n$ matrices  $\{ A_{t_{n}} \}$ such that $ A_{t_{n}} \rightarrow A_t$ in strong operator topology. The finite matrices $A_{t_{n}}$ will be of the form 
 \[  A_{t_{n}}= I_n + t \sum_{j=0}^{n-2} \left(\ket{j}\hspace{-3pt}\bra{j+1} + \ket{j+1}\hspace{-3pt}\bra{j}\right), \]
where $I_n$ is the identity matrix of order $n$ and $\{\ket{j}: 0 \le j \le n-1\}$ is the standard orthonormal basis of $\mathbb{C}^n$, which can also be considered as an element of the aforementioned standard basis of $\ell_2(\mathbb{N})$ truncated at the level $n$. 
It can be noted that the map $\mathcal{B}(\mathbb{C}^n) \ni X \mapsto A_{t_n} \circ X$, defines a trace preserving completely positive map. The operator $A_t$ are Toeplitz operators, and the matrices $A_{t_n}$ are truncated Toeplitz matrices. These operators and their spectral properties are well-studied in literature. Spectrum of $ A_{t_{n}}$ are given by $1+2t\cos \left(\dfrac{k\pi}{n+1}\right)$  for $k=1,2,\ldots,n$ (see \cite{MR3155356}).
$ A_{t_{n}} \rightarrow A_t$ under strong operator topology  \cite{MR1785076}. Hence the spectrum of $ A_{t_{n}}$ converges to spectrum of  $A_t$  which are $1+2t\cos(\pi x)$ where $x\in[0,\infty)$. Clearly $A_t$ is positive if $t\in \left[0,\frac{1}{2} \right]$.
Then the map  $\Lambda_t : \mathcal{B}(\ell_2(\mathbb{N}))\rightarrow \mathcal{B}(\ell_2(\mathbb{N}))$ defined by:
\[ \Lambda_t(X) = A_t\circ X\]
is completely positive and trace-preserving in that range of $t$.  Hence, for $t\in \left[0,\frac{1}{2} \right]$, $ \{\Lambda_t \}$ is family of dynamical maps. 
\par We consider a Hermitian operator $X\in  \mathcal{B}(\ell_2(\mathbb{N})) $ as follows:
\[ X=\begin{bmatrix}
\begin{array}{c|c}
    X_1&O\\ \hline
    O&O\\
    \end{array},
\end{bmatrix}\]
where $X_1$ denotes the $n \times n$ matrices 
\[X_1= \sum_{j=0}^{n-2} \left(\ket{j}\hspace{-3pt}\bra{j+1} + \ket{j+1}\hspace{-3pt}\bra{j}\right).\]
Here $O$ denotes the null matrices of appropriate dimensions. 
Then \[ \Lambda_t(X) = A_t\circ X
= \left[
\begin{array}{c|c}
    B_{t_{n}}&O\\ \hline
    O&O\\
    \end{array} \right],
\]
where  $B_{t_{n}}= t X_1$. 
The spectrum of $\Lambda_t(X)$ is the spectrum of $B_{t_{n}}$ and zero. The spectrum of $B_{t_{n}}$ is given by  
$2t\cos \left(\dfrac{k\pi}{n+1}\right)$, where $k=1,2,\ldots,n$.
Hence, \[\|\Lambda_t(X)\|_1=2t\sum_{k=1}^n \left|\cos\left(\frac{k\pi}{n+1}\right)\right|,\\
\]
where $t\in \left[0,\frac{1}{2} \right]$. Obviously
\[
\frac{\dd}{\dd t}\left\|\Lambda_t(X)\right\|_1=2\sum_{k=1}^n \left|\cos\left(\frac{k\pi}{n+1}\right)\right|> 0
\]
 By \cref{th:diff} $\{\Lambda_t\}$ is not P-divisible.

\par Consider a Hermitian operator $Y\in \mathcal{B}(\ell_2(\mathbb{N}))\otimes\mathcal{B}(\ell_2(\mathbb{N}))$ as follows
\[Y= \left[\begin{array}{c|c}
    X&O\\ \hline
    O&O\\
    \end{array}
\right], \quad \text{where} \; X=\left[
\begin{array}{c|c}
    X_1&O\\ \hline
    O&O\\
    \end{array}
\right].
\]
The natural application of the map \eqref{eg-e1} gives 
\[
(\mathcal{I} \otimes \Lambda_t) (Y) =
\left[
\begin{array}{c|c}
   \Lambda_t (X)&O\\ \hline
    O&O\\
    \end{array}
\right].
\]
Using the same technique, as above, we get the spectrum of $(\mathcal{I} \otimes \Lambda_t) (Y)$ as $2t\cos\left(\frac{k\pi}{n+1}\right)$ and zero, where $k=1,2,\cdots,n.$ Clearly we get the form 
\begin{eqnarray*}
\|(\mathcal{I} \otimes \Lambda_t) (Y)\|_1&= & 2t\sum_{k=1}^n \left|\cos \left(\frac{k\pi}{n+1}\right)\right|\\
\frac{\dd}{\dd t}(\mathcal{I} \otimes \Lambda_t) (Y)\|_1 & =&  2\sum_{k=1}^n \left|\cos\left(\frac{k\pi}{n+1}\right)\right|>0
\end{eqnarray*}
Thus $\{\Lambda_t\}$ is not CP-divisible.

\section{Divisibility of Gaussian channels}\label{gau}

\par In this section we deal with a special type of quantum channels in infinite dimension, viz., Gaussian channels. For completeness we give a very brief introduction of the subject where we fix notations which will be used later in this paper. It should be noted that there are numerous good references in this subject. In particular one can check the book of Holevo \cite{holevo_b} and Serafini \cite{serafini}. For a short introduction of this subject one can check the papers \cite{krpg} and \cite{krprb}. We use the last one here extensively for notations. Furthermore we use the book of Wong \cite{wong-b} and the paper of Caruso et al \cite{holevo-mul} for computational purposes.

\par Let, $\mathcal{H}$  be a finite dimensional Hilbert space. When   $\dim\mathcal{H} =n$ then  $\mathcal{H} = \mathbb{C}^n$. We take element $\bm{\alpha}\in \mathbb{C}^n$ as a column vector $\ket{\bm{\alpha}} = (\alpha_1,\alpha_2,\cdots,\alpha_n )^t$  where $\alpha_i\in \mathbb{C}, \, \forall i=1,2,\cdots,n$  and the inner product between two element $\bm{\alpha}, \,\bm{\beta} \in \mathbb{C}^n$ as follows
\[ \left\langle \bm{\alpha},\bm{\beta} \right\rangle =\sum_{j=1}^n \overline{\alpha_j}  \beta_j, \]
where $\overline{\alpha_j}$ denotes the complex conjugate of $\alpha_j \in \mathbb{C}$. 

Define the Boson Fock space $\Gamma (\mathcal{H})$ over $\mathcal{H}$ by
\[ \Gamma (\mathcal{H})= \mathbb{C}\oplus \mathcal{H}\oplus \mathcal{H}^{\circledS^2} \oplus \cdots \oplus \mathcal{H}^{\circledS^r}  \oplus \cdots  \]
where $\circledS^r$  denotes $r$-fold symmetric tensor product. Elements of the subspace  $\mathcal{H}^{\circledS^r}$ in $\Gamma (\mathcal{H})$ are called $r$-particle vector and elements of the form $u_0 \oplus u_1 \oplus \cdots \oplus u_r \oplus \cdots $, where all but a finite numbers of  $u_r$'s are null, are called finite particle vectors. Finite particle vectors constitute a dense linear manifold $\mathcal{F}$ in $\Gamma (\mathcal{H})$. For any $\bm{u}\in\mathcal{H}$ we associate the exponential vector $e(\bm{u})$ in $\Gamma (\mathcal{H})$  defined by

\[ e(\bm{u})= 1 \oplus \bm{u} \oplus \frac{\bm{u}^{{\otimes}^2}}{\sqrt{2!}} \oplus \cdots \oplus \frac{\bm{u}^{{\otimes}^r}}{\sqrt{r!}} \cdots \]
Then,
\[\left\langle e(\bm{u}),e(\bm{v}) \right\rangle=\exp \left\langle \bm{u},\bm{v} \right\rangle \quad \forall \; \bm{u},\bm{v}\in \mathcal{H}\]
%
%

$S(\mathbb{R}^n)\subset L^2(\mathbb{R}^n) $  be the Schwarz subspace of rapidly decreasing $C^{\infty}$-function on $\mathbb{R}^n$. In $L^2(\mathbb{R}^n)$ one has fundamental momentum and position observable $P_i$ and $Q_i$, $1\leq j \leq n$ which are selfadjoint operators with $S(\mathcal{R}^n)$ as core satisfying the following properties:
\begin{itemize}
\item They form an irreducible family.
\item The operator $P_j$, $1\leq j \leq n$ commute among themselves. So do the operators  $Q_j$, $1\leq j \leq n$. 
\end{itemize}

\par  
We consider an $n$-mode Bosonic system described by the vector of canonical operators
$\hat{r}=\left(Q_1,-P_1,\cdots,Q_n,-P_n\right)^t$ with satisfying canonical commutation relation 
\[ \left[\hat{r},\hat{r}^t  \right]= i\Omega \quad \text{where} \quad \Omega = \bigoplus_{n\text{ copies }} \begin{bmatrix}
0& 1 \\
-1 & 0\\
\end{bmatrix} 
\]
If we re-ordering the canonical operations $\hat{r}$ with the permutation
\[\sigma = \begin{bmatrix}
1&2&3&4&\cdots &2n-1&2n\\
1&n+1&2&n+2&\cdots &n &n+n\\
\end{bmatrix}\]
 Then the canonical operators becomes  $\hat{s}=\left(Q_1,\cdots,Q_n,-P_1,\cdots,-P_n\right)^t$ with satisfying canonical commutation relation becomes,
\begin{equation}\label{3.1}
 \left[\hat{s},\hat{s}^t  \right]= iJ_{2n} \quad \text{with} \quad J_{2n}= \begin{bmatrix}
0 & \mathbb{I}_n\\
-\mathbb{I}_n &0 \\
\end{bmatrix},
\end{equation}
where $\mathbb{I}_n$ represents $n \times n$ identity matrix. 

Let $\rho$ be a state in $L^2(\mathbb{R}^n)$, that is a positive definite trace class operator with $\tr \rho=1 $. Then the complex valued function $\hat{\rho}$ on $\mathbb{C}^n$ defined by
\[\hat{\rho}(\bm{\alpha})= \tr \rho W(\bm{\alpha})  \quad \forall \; \bm{\alpha} =\bm{x}+ \imath \bm{y}=\begin{pmatrix}
\bm{x}\\
\bm{y}\\
\end{pmatrix}, \quad \text{where}\; \bm{x}, \bm{y} \in \mathbb{R}^n.\]
where $W(\bm{\alpha})$ denotes the Wyel operator corresponding to $\bm{\alpha}\in \mathbb{C}^n$, is called the  quantum characteristic function of the state $\rho$.

\begin{defin}
A state $\rho$ acting on  $\Gamma(\mathcal{H})$ with $\mathcal{H}=\mathbb{C}^n$ is called an $n$-mode \emph {Gausssian state} if its  quantum characteristic function $\hat{\rho}$ is given by
\[\hat{\rho}(\bm{x}+\imath \bm{y})=\exp \left[ -\imath\sqrt{2}({\bm{l}}^t\bm{x}-{\bm{m}}^t\bm{y})-{\begin{pmatrix}
\bm{x}\\
\bm{y}\\
\end{pmatrix}}^t
S
\begin{pmatrix}
\bm{x}\\
\bm{y}\\
\end{pmatrix}\right] \]
for all $\bm{x}, \,\bm{y} \in\mathbb{R}^n$, where $\bm{l}, \,\bm{m}$ are element of $\mathbb{R}^n$ and S is $2n\times2n$  symmetric matrix satisfying the matrix inequality
\begin{equation}\label{gm}
 2S+\imath J_{2n} \geq 0,
 \end{equation}
with $J_{2n}$ is defined as in \eqref{3.1}.
\end{defin}
Alternately we may define that a state $\rho$ acting on $\Gamma(\mathbb{C}^n)$ is said to be Gaussian if expectation of $\rho$ with respect to position or momentum operators acting on each mode gives classical Gaussian distribution. Since the position and momentum operators satisfy CCR relation, the covariance matrix written in the order $\left(Q_1,\cdots,Q_n,-P_1,\cdots,-P_n\right)$ will satisfy the \eqref{gm}. Furthermore, a standard result in this area states that, for any $2n \times 2n$ real symmetric strictly positive matrix $S$ satisfying \eqref{gm}, there exists an unique $n$-mode mean zero Gaussian state $\rho_S$ with covariance matrix is $S$. (For calculation of the same one may consult the references  \cite{holevo_b, serafini}.)

\par We now make a few observations on the Weyl operators which will be used for calculations. Consider $\mathbb{R}^n$ as a real Hilbert space. On $L^2(\mathbb{R}^n)$ we define position operator $Q(\bm{h})$ and momentum operator $P(\bm{h})$ as follows:
\begin{eqnarray*} 
Q(\bm{h})f &:& \bm{x} \longrightarrow \left\langle \bm{h}, \bm{x} \right\rangle f(\bm{x}),  \\
P(\bm{h})f &:&  \bm{x} \longrightarrow  \imath \left\langle \bm{h},\nabla_{\bm{x}}\right\rangle f(\bm{x})= \imath\sum h_k \frac{\partial{f(\bm{x})}}{\partial x_k},
 \end{eqnarray*}
 where $\bm{x}=(x_1,x_2,\dots, x_n), \, \bm{h}=(h_1,h_2,\dots h_n),$ and  $\bm{k}=(k_1,k_2,\dots,k_n)\in \mathbb{R}^n $. The corresponding commutator relations take the form, 
 \[  [P(\bm{h}),Q(\bm{k})]f = \imath\left\langle \bm{h},\nabla_{\bm{x}} \right\rangle(\left\langle \bm{k},\bm{x} \right\rangle f)- i\left\langle \bm{k},\bm{x}\right\rangle (\left\langle \bm{h},\nabla_{\bm{x}}\right\rangle f )=\imath\left\langle \bm{h},\bm{k}\right\rangle f,
  \]
  which takes the form 
  \begin{equation}
   \left[P(\bm{h}),Q(\bm{k})\right]=\imath\left\langle \bm{h},\bm{k} \right\rangle \mathrm{Id}, 
   \end{equation}
   Define two unitary operator $U(\bm{h})$ and $V(\bm{h})$ as follows:  
   \begin{alignat*}{2}
   U(\bm{h})&= \exp \left(\imath Q(\bm{h})\right),  \qquad       & U(\bm{h})f(\bm{x})&=\exp\left(\imath \left\langle \bm{h},\bm{x}\right\rangle \right)f(\bm{x}), \text{ and } \\
   V(\bm{h}) &=\exp\left(\imath P(\bm{h})\right) ,        & V(\bm{h})f(\bm{x}) &=\exp \left(-\left\langle \bm{h},\nabla_{\bm{x}}\right\rangle \right)f(\bm{x})=f(\bm{x} - \bm{h}).
\end{alignat*}
  Notice that $U(\bm{h})$ and $V(\bm{h})$ has the following properties,
  \begin{eqnarray*} 
  U(\bm{k})V(\bm{h})f(\bm{x}) &=& \exp \left(\imath \left\langle \bm{h},\bm{x}\right\rangle \right) f(\bm{x}-\bm{h}),  \\
  V(\bm{h})U(\bm{k})f(\bm{x}) &=& V(\bm{h})(U(\bm{k})f(\bm{x})) \\
  &=&V(\bm{h}) (\exp \left(\imath \left\langle \bm{k,x}\right\rangle \right) f(\bm{x})) \\
  &=& \exp \left(\imath \left\langle \bm{k,x-h}\right\rangle \right) f\bm{(x-h)} \\
  &=& \exp \left(-\imath \left\langle \bm{k,h}\right\rangle \right) U(\bm{k})V(\bm{h})f(\bm{x}),\\
  U(\bm{k})V(\bm{h})&=& \exp \left(\imath \left\langle \bm{k,h}\right\rangle \right) V(\bm{h})U(\bm{k}).
  \end{eqnarray*}
 We define Weyl displacement operator as 
  \begin{equation}
  W(\bm{k,h}):= U(\bm{k})V(\bm{h}),
  \end{equation}
which acts on the function $f \in L^2(\mathbb{R}^n) $  as follows:
  \[(W(\bm{k},\bm{h})f)(\bm{x})= \exp \left(\imath \left\langle \bm{k,h}\right\rangle \right) f(\bm{x} - \bm{h}).  \]
 Foe $\bm{\alpha}_1 = \bm{k}_1 + \imath \bm{h}_1$ and $\bm{\alpha}_2 = \bm{k}_2 + \imath \bm{h}_2 \in \mathbb{C}^n$ the product of Weyl operators is given as, 
  \begin{eqnarray*}
   W(\bm{k}_1, \bm{h}_1)W(\bm{k}_2 , \bm{h}_2)&=& U(\bm{k}_1)V(\bm{h}_1)U(\bm{k}_2)V(\bm{h}_2) \\
   &=& \exp \left(-\imath \left\langle \bm{h}_1,\bm{k}_2\right\rangle \right) U(\bm{k}_1)U(\bm{k}_2)V(\bm{h}_1)V(\bm{h}_2)\\
   &=& \exp \left(-\imath \left\langle \bm{\bm{h}_1,\bm{k}_2}\right\rangle \right) U(\bm{k_2})U(\bm{k_1})V(\bm{h_2})V(\bm{h_1})\\
   &=& \exp \left(-i(\left\langle \bm{k_1,h_2}\right\rangle -\left\langle \bm{h_1,k_2}\right\rangle ) \right)W(\bm{k_2,h_2})W(\bm{k_1,h_1})
\end{eqnarray*} 
\[ W(\bm{k_1,h_1})W(\bm{k_2,h_2}) = \exp \left(-\imath (\left\langle \bm{k_1,h_2} \right\rangle -\left\langle \bm{h_1,k_2}\right\rangle )\right)W(\bm{k_2,h_2})W(\bm{k_1,h_1})\]
which is known as Weyl commutation relation. By Weyl functional calculus, for  $\phi \in L^2(\mathbb{R}^n \times \mathbb{R}^n)$, $W(\phi)$ can be defined as
\[W(\phi)= \int_{\mathbb{R}^n} \int_{\mathbb{R}^n} \phi(\bm{k},\bm{h})W(\bm{k},\bm{h}) \,\dd\bm{k} \,\dd\bm{h}. \]
Action of this can be written as - 
\begin{eqnarray*}
(W(\phi)f)(\bm{x})&=& \int_{\mathbb{R}^n} \int_{\mathbb{R}^n} \phi(\bm{k},\bm{h})\exp \left(\imath \left\langle \bm{k},\bm{x} \right\rangle\right)f(\bm{x}-\bm{h}) \,\dd\bm{k} \,\dd\bm{h} \\
&=& \int_{\mathbb{R}^n} \exp \left( \imath \left\langle \bm{k},\bm{x} \right\rangle \right)\left(\int_{\mathbb{R}^n} \phi(\bm{k},\bm{h})f(\bm{x}-\bm{h}) \,\dd\bm{h}\right) \,\dd\bm{k} \\
&=& \int_{\mathbb{R}^n} \left(\int_{\mathbb{R}^n} \exp \left(\imath \left\langle \bm{k,x} \right\rangle \right) \phi(\bm{k},\bm{x}-\bm{l}) \,\dd\bm{k}\right)f(\bm{l}) \,\dd\bm{l} \quad \textit{where} \; \bm{h}=\bm{x}-\bm{l} , \,\dd\bm{h}= -\,\dd\bm{l} \\
&=&  \int_{\mathbb{R}^n} K_{\phi}(\bm{x},\bm{l})f(\bm{l})\,\dd\bm{l},
\end{eqnarray*}
where $ K_{\phi}(\bm{x},\bm{l}):= \int_{\mathbb{R}^n} \exp \left( \imath \left\langle \bm{k},\bm{x} \right\rangle \right) \phi(\bm{k},\bm{x}-\bm{l}) \,\dd\bm{k} = (\mathcal{F}_1 \phi)(\bm{x},\bm{x}-\bm{l})$. 
$(\mathcal{F}_1 \phi)(\bm{x},\bm{x}-\bm{l})$ is the Fourier transformation in the first variable. After suitable  change of variables
\[g:(\bm{x},\bm{l}) \longrightarrow (\bm{x},\bm{x}-\bm{l})\] 
with the Jacobean given by 
\[J_g =\begin{bmatrix}
1 &0\\
1 &-1 \\
\end{bmatrix}
\] 
where $|J_g|=1$. We compute the trace of  $W(\phi)$ by Mercer's theorem if $K_{\phi}$ is regular. (See for instance \cite{MR0889455, MR0883367}, or \cite[Theorem 1.1]{MR2501172} for a short proof.)
\begin{eqnarray*}
\tr W(\phi)=\tr K_\phi &=& \int_{\mathbb{R}^n} K_{\phi}(\bm{x},\bm{x})\,\dd\bm{x} \\
&=& \int_{\mathbb{R}^n} \int_{\mathbb{R}^n} \exp \left(\imath \left\langle \bm{k},\bm{x}\right\rangle \right) \phi(\bm{k},\bm{x}-\bm{x}) \,\dd\bm{k}\,\dd\bm{x}\\
&=& \int_{\mathbb{R}^n} (\mathcal{F}_1 \phi)(\bm{x},\bm{0}) \,\dd\bm{x} \\
&=& \mathcal{F}_1^{-1} o \mathcal{F}_1 (\phi)|_{(\bm{0},\bm{0})}
\end{eqnarray*}
Thus
\[  \tr W(\phi) = \frac{1}{2\pi^n} \phi(\bm{0},\bm{0}).\]

\par Let $\mathcal{H}=L^2(\mathbb{R}^n)$ as earlier  and $T:\mathcal{B}(\mathcal{H})\longrightarrow \mathcal{B}(\mathcal{H})$ is a Gaussian channel. Then by Heinosaari, Holevo, Wolf(2009) \cite{holevo2009}, we can label Gaussian channels by a pair $(X,Y)$ of real $2n\times2n$ matrices satisfying the inequality
\begin{equation}\label{gieq} 
Y\geq i(J-X^tJX) \quad \text{with} \quad J= \begin{bmatrix}
0 & \mathbb{I}_n\\
-\mathbb{I}_n &0 \\
\end{bmatrix}
\end{equation}
The corresponding Gaussian channel acts on Weyl operators as,
\begin{eqnarray*}
 T_{X,Y}(W(\bm{k},\bm{h})) &=& \exp \frac{\imath}{2} \left\langle \begin{bmatrix} \bm{k}\\ \bm{h}\end{bmatrix} ,  Y\begin{bmatrix} \bm{k}\\ \bm{h}\end{bmatrix} \right \rangle W \left(X  \begin{bmatrix} \bm{k}\\ \bm{h}\end{bmatrix} \right)\\
T_{X,Y}W(\phi)
&=& T_{(X,Y)}\left( \int_{\mathbb{R}^n} \int_{\mathbb{R}^n} \phi(\bm{k},\bm{h})W(\bm{k},\bm{h}) \,\dd\bm{k} \,\dd\bm{h} \right) \\
&=& \left( \int_{\mathbb{R}^n} \int_{\mathbb{R}^n} \phi(\bm{k},\bm{h})T_{(X,Y)}(W(\bm{k},\bm{h})) \,\dd\bm{k} \,\dd\bm{h} \right) \\
&=& \int_{\mathbb{R}^n} \int_{\mathbb{R}^n} \phi(\bm{k},\bm{h})\exp \frac{\imath}{2} \left\langle \begin{bmatrix} \bm{k}\\ \bm{h}\end{bmatrix} ,  Y\begin{bmatrix} \bm{k}\\ \bm{h}\end{bmatrix} \right \rangle W \left(X  \begin{bmatrix} \bm{k}\\ \bm{h}\end{bmatrix} \right) \,\dd\bm{k} \,\dd\bm{h} \\
&&\text{change of variable:} \begin{bmatrix} \bm{u}\\\bm{v} \\ \end{bmatrix}=X\begin{bmatrix} \bm{k}\\ \bm{h} \\ \end{bmatrix}, \text{ assuming X is invertible.} \\
&& \,\dd\bm{u} \,\dd\bm{v} = \sqrt{\det(X^TX)}\,\dd\bm{k} \,\dd\bm{h} \\
&=&\int_{\mathbb{R}^n} \int_{\mathbb{R}^n} \phi\left(X^{-1}\begin{bmatrix} \bm{u}\\ \bm{v} \\ \end{bmatrix}\right)\exp \left[{-\frac{\imath}{2}\left\langle \begin{bmatrix}
\bm{u}\\ \bm{v} \\ \end{bmatrix},{X^T}^{-1}YX^{-1}\begin{bmatrix}
\bm{u}\\ \bm{v} \\ \end{bmatrix}\right\rangle} \right] \\
&&\hspace{10pt}W\left(  \begin{bmatrix} \bm{u}\\ \bm{v} \\ \end{bmatrix}\right) \sqrt{\det(X^TX)}\,\dd\bm{u} \,\dd\bm{v} \\
&=& W(S_{X,Y}(\phi))
\end{eqnarray*} 
where 
\[S_{X,Y}(\phi)=\exp\left[ {-\frac{\imath}{2}\left\langle \begin{bmatrix}
\bm{u}\\ \bm{v} \\ \end{bmatrix},{X^T}^{-1}YX^{-1}\begin{bmatrix}
\bm{u}\\ \bm{v} \\ \end{bmatrix}\right\rangle} \right]
\phi\left(X^{-1}\begin{bmatrix} \bm{u}\\ \bm{v} \\ \end{bmatrix}\right) \sqrt{\det(X^TX)} .\] 
This implies that (ignoring the scalar factor $1/(2\pi^n)$) whenever $\tr (W(\phi))= \phi(0,0)$, we have 
\begin{eqnarray*}
\tr (W(S_{X,Y}(\phi)))&=& \tr (T_{X,Y}(W(\bm{k},\bm{h}))\\
&=& S_{X,Y}(\phi)(\bm{0,0}) \\
&=& \sqrt{\det(X^TX})\phi(\bm{0,0}) \\
&=& \sqrt{\det(X^TX)} \tr(W(\phi))\\
\tr (W(S_{X,Y}(\phi)))&=& \sqrt{\det(X^TX)} \tr(W(\phi))= \phi(\bm{0},\bm{0})\det X
\end{eqnarray*}
Let $T_{X_t,Y_t}$ be a gaussian channel and
$W(\phi)$ a Weyl operator such that $W(\phi)$ is positive, i.e., with $\phi$ a positive definite kernel. Then $T_{X_t,Y_t}(W(\phi))$ is positive, and
\[ \|T_{X_t,Y_t}(W(\phi))\|_1= \tr T_{X_t,Y_t}(W(\phi))= \phi (\bm{0},\bm{0})\det X_t  \]
Since $\phi$ is positive definite, $\phi (\bm{0},\bm{0}) >0$.
The derivative of the $1$-norm is given by
\[\frac{\dd}{\dd t} \| T_{X_t,Y_t}(W(\phi)) \|_1 = \phi (\bm{0},\bm{0})\frac{\dd}{\dd t} \det X_t. \]
Thus we have proved the following theorem. 
\begin{theorem}\label{th4.1}
Let $\{T_{X_t,Y_t}\}$ be a family of P-divisible Gaussian channels where $X_t$ is invertible. Then 
 \[
 \frac{\dd}{\dd t}\left(\det X_t \right)  \leq 0.
 \]
\end{theorem}

\par The above theorem can be exploited for construction of Gaussian dynamical maps which are not P-divisible.  Choose $X_t$ in the Gaussian channel such that $\dfrac{\dd}{\dd t} \det X_t >0 $. Then,
\[ \frac{\dd}{\dd t} \| T_{X_t,Y_t}(W(\phi)) \|_1 >0 \] 
Which shows that $T_{X_t,Y_t}$ not P-divisble.

\par  For not CP-divisible case,
we take any hermitian operator $A\in \mathcal{B}_1(\mathcal{H}\otimes\mathcal{H}) $ as a difference of two positive operator $A_1$ and $A_2$, where $A_1= W(\phi_1)\otimes B_1$ and $A_2= W(\phi_2) \otimes B_2$ with $B_1,B_2\geq 0$ and $B_1B_2=0$. $\phi_1, \phi_2$  are positive definite kernel.
\begin{eqnarray*}
\|T_{X_t,Y_t} \otimes \mathcal{I} \left( W(\phi_1)\otimes B_1-W(\phi_2)\otimes B_2\right)\|_1 &=&
 \|T_{X_t,Y_t} \otimes \mathcal{I} ( W(\phi_1)\otimes B_1)- T_{X_t,Y_t} \otimes \mathcal{I}(W(\phi_2)\otimes B_2)\|_1\\
&=& \|T_{X_t,Y_t}( W(\phi_1))\otimes B_1- T_{X_t,Y_t}(W(\phi_2))\otimes B_2\|_1 \\
& =&  \tr T_{X_t,Y_t}( W(\phi_1))\otimes B_1 + \tr T_{X_t,Y_t}( W(\phi_2))\otimes B_2\\
&= & \tr T_{X_t,Y_t}( W(\phi_1))\tr B_1 + \tr T_{X_t,Y_t}( W(\phi_2))\tr B_2 \\
&=& \det X_t \phi_1(\bm{0},\bm{0})\tr B_1 + \det X_t \phi_2(\bm{0},\bm{0})\tr B_2\\
&=& \det X_t (\phi_1(\bm{0},\bm{0})\tr B_1 + \phi_2(\bm{0},\bm{0})\tr B_2).
\end{eqnarray*}
Thus the derivative of $1$-norm takes the following form
\[ \frac{\dd}{\dd t}\|T_{X_t,Y_t} \otimes \mathcal{I} ( W(\phi_1)\otimes B_1-W(\phi_2)\otimes B_2)\|_1 
 = \left(\phi_1(\bm{0},\bm{0})\tr B_1 + \phi_2(\bm{0},\bm{0})\tr B_2\right) \frac{\dd}{\dd t}\det X_t.\]
 It has been seen that  $\phi_1(\bm{0},\bm{0})\tr B_1 + \phi_2(\bm{0},\bm{0})\tr B_2 \geq 0$. CP-divisibility of Gaussian channel implies that $\frac{\dd}{\dd t}\left(\det X_t \right)  \leq 0$ .

\par Any Gaussian channel has a unitary dilation. In fact it can be shown that such a dilation will be Gaussian unitary, which means, the unitary will have a representation in terms of symplectic matrices. Caruson et al \cite{holevo-mul} showed that an $n$-mode bosonic Gaussian channel will have a Gaussian unitary dilation in $2n$-mode system. We use the result to  construct a Gaussian channel in the covariance matrix level transformation which is given below. 
\begin{eg} \label{eg4.1}
\par Let $R_1, R_2\in Sp(4)$, where $Sp(4)$ is group of symplectic matrix of order $4$, as follows:
\[ R_1=\begin{bmatrix}
X_1& Y_1\\
-Y_1 & X_1\\
\end{bmatrix} \quad \text{and} \quad  
R_2=\begin{bmatrix}
X_2& Y_2\\
-Y_2 & X_2\\
\end{bmatrix},
\]
with
 \[ X_1= \frac{1}{\sqrt{2}}\begin{bmatrix}
1& 0\\
0 & 1\\
\end{bmatrix}, \;
 Y_1= \frac{1}{\sqrt{2}}\begin{bmatrix}
0& 1\\
-1 & 0\\
\end{bmatrix}, \;
 X_2= \frac{1}{2}\begin{bmatrix}
1 & 1\\
1 & 1\\
\end{bmatrix},\;
 Y_2= \frac{1}{2}\begin{bmatrix}
1 & -1\\
-1 & 1\\
\end{bmatrix}.
\]
Let $T$ be the diagonal matrix defined as  $\mathrm{diag}\left(1,t,1,\dfrac{1}{t}\right)$. Define 
 \[L=R_1TR_2=\frac{1}{2\sqrt{2}} \left[
\begin{array}{cc|cc}
   1+\frac{1}{t} & 1-\frac{1}{t}& 1+\frac{1}{t} & -1+\frac{1}{t}\\
t+1 & t-1 & -t-1 & t-1 \\ \hline
-t-1 & -t+1 & t+1 & -t+1 \\
1+\frac{1}{t} & 1-\frac{1}{t}& 1+\frac{1}{t} & -1+\frac{1}{t}\\
    \end{array} \right] \] 
After reordering position and momentum vector we get,
\[L^{(1)}=\frac{1}{2\sqrt{2}}\left[
\begin{array}{cc|cc}
   1+\frac{1}{t} & 1+\frac{1}{t}& 1-\frac{1}{t} & -1+\frac{1}{t}\\
-t-1 & t+1 & -t+1 & _t+1 \\ \hline
t+1 & -t-1 & t-1 & t-1 \\
1+\frac{1}{t} & 1+\frac{1}{t}& 1-\frac{1}{t} & -1+\frac{1}{t}\\
    \end{array} \right]=
    \left[
\begin{array}{c|c}
    L_{11}&L_{12}\\ \hline
   L_{21}&L_{22}\\
    \end{array}
\right]\]
Now, take the transformation
\[\Psi: \left[
\begin{array}{c|c}
    S&0\\ \hline
   0&\frac{1}{2} \mathbb{I}_2\\
    \end{array}
\right] \longrightarrow
L^{(1)}\left[
\begin{array}{c|c}
    S&0\\ \hline
   0&\frac{1}{2} \mathbb{I}_2\\
    \end{array}
\right]{L^{(1)}}^t
= \left[
\begin{array}{c|c}
   L_{11}SL_{11}^t+\frac{1}{2}L_{12}L_{12}^t&\hspace{5mm} * \hspace{5mm}\\ \hline
   *&\hspace{5mm}*\hspace{5mm}
    \end{array}
\right]
\] 
Taking partial trace and making the map from $1$-mode to $1$-mode is to restrict the map on first component, which gives the Gaussian channel
\[ \Phi: S \longrightarrow  L_{11}SL_{11}^t+\frac{1}{2}L_{12}L_{12}^t. \]
Notice that we are getting a time dependent $1$-parameter family of  Gaussian channels, i.e. a Gaussian dynamical map.   In this case $X_t=L_{11}$ and $Y_t=L_{12}L_{12}^t$ and the $t$-parametrised family of Gaussian channel is given by $\left\{\Phi_{X_t,Y_t} \right\}$. By construction of Caruso et al \cite{holevo-mul} this will also satisfy the inequality \eqref{gieq}.  

Let $W(\phi)$ be Weyl operator in such a way that $W(\phi)$ is positive i.e., with $\phi$ a positive definite kernel. Then $\Phi_{X_t,Y_t}(W(\phi))$ is positive.
\[ \|\Phi_{X_t,Y_t}(W(\phi))\|_1= \tr \Phi_{X_t,Y_t}(W(\phi))= \phi (\bm{0},\bm{0})\det X_t.  \]
Since $\phi$ is positive definite kernel, $\phi (\bm{0},\bm{0}) >0$.
\[ \| \Phi_{X_t,Y_t}(W(\phi)) \|_1 = \phi (\bm{0},\bm{0}) \det X_t=2(1+t)(1+\frac{1}{t}) \phi (\bm{0},\bm{0}),\]
and its derivative is 
\[ \frac{\dd}{\dd t} \| \Phi_{X_t,Y_t}(W(\phi)) \|_1=2 \frac{\left(1+t\right)}{t^2} (t-1)\phi(\bm{0},\bm{0})= 2 \frac{\left(t^2-1\right)}{t^2}>0, \] 
whenever $t^2>1$. Hence $\Phi$ is not P-divisible by \cref{th4.1}.
\end{eg}

\begin{eg}\label{eg4.2}
   Construct a Gaussian channel in the covariance matrix level transformation.
Define  $R_1, R_2\in Sp(6)$ as follows: 
\[ R_1=\begin{bmatrix}
X_1& Y_1\\
-Y_1 & X_1\\
\end{bmatrix} \quad \text{and} \quad  
R_2=\begin{bmatrix}
X_2& Y_2\\
-Y_2 & X_2\\
\end{bmatrix},
\]
where  
\begin{alignat*}{2}
X_1&= \frac{1}{\sqrt{3}}\begin{bmatrix}
1&1&1\\
1& -\frac{\sqrt{3}}{2} & -\frac{\sqrt{3}}{2}\\
1&-\frac{\sqrt{3}}{2}&-\frac{\sqrt{3}}{2}\\
\end{bmatrix}, \qquad
 &Y_1 &= \frac{1}{\sqrt{3}}\begin{bmatrix}
0& 0&0\\
0 & \frac{1}{2}&-\frac{1}{2}\\
0&-\frac{1}{2} &\frac{1}{2}\\
\end{bmatrix}, \\
 X_2 &= \begin{bmatrix}
\frac{1}{\sqrt{2}} & \frac{1}{\sqrt{2}}&0\\
\frac{1}{\sqrt{2}} & \frac{1}{\sqrt{2}}&0\\
0&0&1\\
\end{bmatrix},
& Y_2 &= \frac{1}{\sqrt{2}}\begin{bmatrix}
1 & -1&0\\
-1 & 1&0\\
0&0&0 \\
\end{bmatrix}.
\end{alignat*}
As earlier, define a diagonal matrix $T=\mathrm{diag} \left(1,t,t^2,1,\dfrac{1}{t},\dfrac{1}{t^2}\right)$.
Using the notations of the previous example we define 
\begin{eqnarray*}
L &=&R_1TR_2 \\
&=& \frac{1}{\sqrt{6}}\left[
\begin{array}{cccc|cc}
   1+t & 1+t&\sqrt{2}t^2&1-t&-1+t&0 \\
1-\frac{\sqrt{3}}{2}t+\frac{1}{2t}& 1-\frac{\sqrt{3}}{2}t-\frac{1}{2t}& -\frac{\sqrt{6}}{2}t^2&1+\frac{\sqrt{3}}{2}t+\frac{1}{2t}&-1-\frac{\sqrt{3}}{2}t-\frac{1}{2t}&-\frac{1}{\sqrt{2}t^2}\\
1-\frac{\sqrt{3}}{2}t-\frac{1}{2t}&1-\frac{\sqrt{3}}{2}t+\frac{1}{2t}&-\frac{\sqrt{6}}{2}t^2&1+\frac{\sqrt{3}}{2}t-\frac{1}{2t}&-1-\frac{\sqrt{3}}{2}t-\frac{1}{2t}&\frac{1}{\sqrt{2}t^2} \\ 
-1+\frac{1}{t} & 1-\frac{1}{t}& 0 & 1+\frac{1}{t} &1+\frac{1}{t}&\frac{\sqrt{2}}{t^2} \\\hline
-\frac{t}{2}-1-\frac{\sqrt{3}}{2t} & -\frac{t}{2}+1-\frac{\sqrt{3}}{2t}& \frac{t^2}{\sqrt{2}} & \frac{t}{2}+1-\frac{\sqrt{3}}{2t}&-\frac{t}{2}+1-\frac{\sqrt{3}}{2t}&-\frac{\sqrt{3}}{\sqrt{2}t^2}\\
\frac{t}{2}-1-\frac{\sqrt{3}}{2t}&\frac{t}{2}+1-\frac{\sqrt{3}}{2t}&-\frac{t^2}{\sqrt{2}}&-\frac{t}{2}+1-\frac{\sqrt{3}}{2t}&\frac{t}{2}+1-\frac{\sqrt{3}}{2t}&-\frac{\sqrt{3}}{\sqrt{2}t^2}\\
    \end{array} 
    \right].
    \end{eqnarray*}
After reordering position and momentum vector we get,
\begin{eqnarray*}
L^{(1)} &=& \frac{1}{\sqrt{6}}\left[
\begin{array}{cccc|cc}
   1+t & 1-t& 1+t & -1+t& \sqrt{2}t^2&0\\
-1+\frac{1}{t} & 1+\frac{1}{t} & 1-\frac{1}{t} & -1+\frac{1}{t}&0&\frac{\sqrt{2}}{t^2} \\ 
1-\frac{\sqrt{3}}{2}t+\frac{1}{2t} & 1+\frac{\sqrt{3}}{2}t+\frac{1}{2t} & 1-\frac{\sqrt{3}}{2}t-\frac{1}{2t} & -1-\frac{\sqrt{3}}{2}t+\frac{1}{2t}& -\frac{\sqrt{6}}{2}t^2& -\frac{1}{\sqrt{2}t^2} \\
-\frac{t}{2}-1-\frac{\sqrt{3}}{2t} & \frac{t}{2}+1-\frac{\sqrt{3}}{2t}& -\frac{t}{2}+1-\frac{\sqrt{3}}{2t} & -\frac{t}{2}+1-\frac{\sqrt{3}}{2t} & \frac{t^2}{\sqrt{2}}&-\frac{\sqrt{3}}{\sqrt{2}t^2}\\\hline
1-\frac{\sqrt{3}}{2}t-\frac{1}{2t}&1+\frac{\sqrt{3}}{2}t-\frac{1}{2t}&1-\frac{\sqrt{3}}{2}t+\frac{1}{2t}&-1-\frac{\sqrt{3}}{2}t-\frac{1}{2t}&-\frac{\sqrt{6}}{2}t^2&\frac{1}{\sqrt{2}t^2}\\
\frac{t}{2}-1-\frac{\sqrt{3}}{2t}&-\frac{t}{2}+1-\frac{\sqrt{3}}{2t}&\frac{t}{2}+1-\frac{\sqrt{3}}{2t}&\frac{t}{2}+1-\frac{\sqrt{3}}{2t}&-\frac{t^2}{\sqrt{2}}&-\frac{\sqrt{3}}{\sqrt{2}t^2}\\

    \end{array} \right] \\
    &=&
    \begin{bmatrix}
\begin{array}{c|c}
    L_{11}&L_{12}\\ \hline
   L_{21}&L_{22}\\
    \end{array}
\end{bmatrix}
\end{eqnarray*}
Now, take the Gaussian unitary transformation in the level of covariance matrices as follows: 
\[\Psi: \begin{bmatrix}
\begin{array}{c|c}
    S&0\\ \hline
   0&\frac{1}{2} \mathbb{I}_2
    \end{array}
\end{bmatrix} \longrightarrow
L^2\begin{bmatrix}
\begin{array}{c|c}
    S&0\\ \hline
   0&\frac{1}{2} \mathbb{I}_2
    \end{array}
\end{bmatrix}{L^2}^t
= \begin{bmatrix}
\begin{array}{c|c}
   L_{11}SL_{11}^t+\frac{1}{2}L_{12}L_{12}^t& \hspace{5pt}*\hspace{5pt}\\ \hline
  *  &\hspace{5pt}*\hspace{5pt}
    \end{array}
\end{bmatrix}.
\] 
As earlier, if we restrict the map on first component then we get 
\[ \Phi_{X_t,Y_t}: S \longrightarrow  L_{11}SL_{11}^t+\frac{1}{2}L_{12}L_{12}^t \]
is a Gaussian channel from $2$-mode to $2$-mode. In this case $X_t=L_{11}$ and $Y_t=L_{12}L_{12}^t$.
$W(\phi)$ be Weyl operator in such a way that $W(\phi)$ is positive i.e., with $\phi$ a positive definite kernel. Then $\Phi_{X_t,Y_t}(W(\phi))$ is positive.
\[ \|\Phi_{X_t,Y_t}(W(\phi))\|_1= \tr \Phi_{X_t,Y_t}(W(\phi))= \phi (\bm{0},\bm{0})\det X_t  \]
Since $\phi$ is positive definite kernel, $\phi (\bm{0},\bm{0}) >0$.
\begin{eqnarray*} 
\| \Phi_{X_t,Y_t}(W(\phi)) \|_1 &=& \phi (\bm{0},\bm{0}) \det X_t \\
&=& \left( 80+24\sqrt{2}+16\sqrt{3}+32\sqrt{6}+\frac{2\sqrt{3}}{t^2}-\frac{6(\sqrt{2}+\sqrt{3})}{t}\right)\phi (\bm{0},\bm{0}).
\end{eqnarray*}
We get the derivative
\[ \frac{\dd}{\dd t} \| \Phi_{X_t,Y_t}(W(\phi)) \|_1=\frac{6(\sqrt{2}+\sqrt{3})t-4\sqrt{3}}{t^3}\phi(\bm{0},\bm{0})>0, \quad \text{ whenever ~~} t> \frac{2}{\sqrt{3}(\sqrt{2}+\sqrt{3})} \approx 0.367. \] 
Hence, $\Phi$ is not P-divisible by \cref{th4.1}.

\end{eg}

\section{Conclusion}\label{s5}
In this paper we have shown that if a family of dynamical map on trace class operator $\mathcal{B}_1(\mathcal{H})$, where $\mathcal{H}$  is a separable Hilbert space is P-divisible (resp. CP-divisible), then it is sufficient to show that derivative of $\|\Lambda_t(X)\|_1$ (resp. $ \| \mathcal{I} \otimes \Lambda_t (Y) \|_1$)   is not positive  for any choice of Hermitian operator $X$ (resp. $Y$) in suitable Hilbert space provided the derivatives exist. Converse of the theorem holds whenever the family of dynamical maps are surjective. We have applied this theorem on various families  of dynamical maps, viz.  unitary evolution, idempotent maps, Gaussian channels and have checked their P-divisibility as well as  CP-divisibility. For one parameter family of combination of idempotent maps we construct such conditions for which the family is P or CP-divisible. 
 \par In a special case we obtained an obstruction to P-divisibility of an  $n$-mode Gaussian channel $\{\Phi_{X_t,Y_t}\}$ using the derivative of the determinant of $X_t$. We have given explicit examples of such channels.  
 \par Converse of our result  depends on surjectivity  of the dynamical maps. We conjecturise that this result holds for any family of dynamical maps (i.e. one parameter CPTP maps) as well when the underlying Hilbert space $\mathcal{H}$ is of infinite dimension. As expected, this holds when $\mathcal{H}$ is of finite dimensions \cite{Rivas2011, CRS-prl-18}.  In general for one parameter  CP maps on $\mathcal{B}_1(\mathcal{H})$, this may not be true. 
 
\section*{Acknowledgements} 
SP is supported by \href{https://www.csir.res.in/}{{\sf CSIR, Govt. of India}} research fellowship.
RS acknowledges financial support from \href{https://dst.gov.in}{DST}, Govt. of India, project number  {\sf DST/ICPS/QuST/Theme- 2/2019/General Project Q-90}.  SP and RS thank Prof. Rajendra Bhatia for useful comments and suggestions.
\bibliographystyle{acm}
\bibliography{biblio}
\end{document}